\documentclass[11pt]{amsart}

\usepackage{amssymb}
\usepackage{amscd}
\usepackage{amsmath}
\usepackage{amssymb}
\usepackage{amsfonts}
\usepackage{amsmath}
\usepackage{amscd}
\usepackage{amsthm}
\usepackage[all,cmtip]{xy}
\usepackage[all]{xy} 
\usepackage{bm}
\usepackage[all]{xy}
\usepackage{color}
\usepackage[hmargin=3cm,tmargin=4cm,bmargin=3cm]{geometry}
\newtheorem{theorem}{Theorem}[section]
\newtheorem{lemma}[theorem]{Lemma}
\newtheorem{prop}[theorem]{Proposition}
\newtheorem{coro}[theorem]{Corollary}

\theoremstyle{remark}
\newtheorem{remark}[theorem]{Remark}

\newtheorem{definition}[theorem]{Definition}

\newcommand{\Gal}{\mathrm{G}}
\newcommand{\Aut}{\mathrm{Aut}}
\newcommand{\Hom}{\mathrm{Hom}}
\newcommand{\sgn}{\mathrm{sgn}}

\newcommand{\F}{\mathbb{F}}

\newcommand{\fra}{\mathfrak{a}}
\newcommand{\frb}{\mathfrak{b}}

\newcommand{\fre}{\mathfrak{e}}
\newcommand{\frf}{\mathfrak{f}}

\renewcommand{\frm}{\mathfrak{m}}
\newcommand{\NN}{\mathbb{N}}

\newcommand{\Ce}{{\mathbb{C}}}
\newcommand{\Qp}{{\mathbb{Q}_p}}

\newcommand{\Ze}{\mathbb{Z}}
\newcommand{\Zl}{{\mathbb{Z}_\ell}}
\newcommand{\Zp}{\mathbb{Z}_p}
\newcommand{\sseq}{\subseteq}
\newcommand{\Fp}{\mathbb{F}_p}
\newcommand{\Fq}{\mathbb{F}_q}
\newcommand{\Finfty}{\mathbb{F}_\infty}
\newcommand{\frp}{\mathfrak{p}}

\newcommand{\lra}{\longrightarrow}
\newcommand{\ra}{\rightarrow}

\newcommand{\Pic}{\mathrm{Pic}}
\renewcommand{\div}{\mathrm{div}}
\newcommand{\ZG}{\Ze[G]}

\newcommand{\ZlG}{\Zl[G]}

\newcommand{\Gl}{\mathrm{GL}}

\newcommand{\calK}{\mathcal{K}}
\newcommand{\calL}{\mathcal{L}}

\newcommand{\calO}{\mathcal{O}}

\newcommand{\tensor}{\otimes}

\newcommand{\Fitt}{\mathrm{Fitt}}

\newcommand{\logIw}{\log_\mathrm{Iw}}
\newcommand{\AIw}{A_\mathrm{Iw}}
\newcommand{\nt}{\sseq}
\newcommand{\MSl}{M_S^{(n)}}
\newcommand{\Tpl}{T_p^{(n)}}
\newcommand{\TpMSl}{T_p\left(M_S^{(n)}\right)}

\newcommand{\Nabla}{\bigtriangledown}



\date{}

\begin{document}

\title[Geometric Main Conjectures in Function Fields]{Geometric main conjectures in function fields}

\author{Werner Bley} 

\author{Cristian D. Popescu$^\ast$}\thanks{$^\ast$Supported by a Simons Foundation collaboration grant}

\maketitle

\begin{abstract} We prove an Equivariant  Main Conjecture in Iwasawa Theory along any rank one, sign-normalized Drinfeld modular, split at $\infty$ Iwasawa tower of a general function field 
of characteristic $p$, for the Iwasawa modules recently considered by Greither and Popescu in \cite{GP12}, in their proof of the classical Equivariant Main Conjecture 
along the (arithmetic) cyclotomic Iwasawa tower. As a consequence, we prove an Equivariant Main Conjecture for a projective limit of certain Ritter--Weiss type modules, along the same Drinfeld modular Iwasawa towers.  This generalizes the results of Angl\`es et.al. \cite{Angles}, Bandini et al. \cite{Bandini}, and Coscelli \cite{Coscelli}, for the split at $\infty$ piece of the Iwasawa towers considered in loc.cit.,  and refines the results in \cite{GP12}.
\end{abstract}

%
\section{Introduction and Notations}\label{intro}

\subsection{Arithmetic Iwaswa Theory} In \cite{GP12}, Greither and the second author considered a set of data $(K/k, S, \Sigma)$ consisting of an abelian extension $K/k$ of global fields of characteristic $p>0$ of Galois group $G$ and two finite, nonempty, disjoint sets of places $S$ and $\Sigma$ in $k$, such that $S$ contains the ramification locus of $K/k$. From this data one can construct a Deligne--Picard $1$--motive $M_{S,\Sigma}$, which is naturally acted upon by the Galois group $G\times\mathbf{\Gamma}$, where $\Gamma:=G(\overline{\Bbb F_q}/\Bbb F_q)$ and $\Bbb F_q$ is the exact field of constants of $k$.  As a consequence, all the $\ell$--adic realizations $T_{\ell}(M_{S,\Sigma})$ are natural finitely generated modules over the profinite group--algebra $\Bbb Z_{\ell}[[G\times\Gamma]]$, for all prime numbers $\ell$. 

On the other hand, the set of data $(K/k,S,\Sigma)$ gives rise to a polynomial $\Theta_{S,\Sigma}(u)\in\Bbb Z[G][u]$, which is uniquely determined by the packet of ($S$--incomplete, $\Sigma$--smoothed) Artin $L$--functions $L_{S,\Sigma}(\chi, s)$, for all the complex valued characters $\chi$ of $G$,
via the equalities
\[
  \chi(\Theta_{S,\Sigma}(u))\mid_{u=q^{-s}}=L_{S, \Sigma}(\chi^{-1}, s),
\]
for all $s\in\Bbb C$.
The main result in \cite{GP12} is the following $G$--equivariant Iwasawa main conjecture, along the arithmetic (cyclotomic) Iwasawa tower $(K\otimes_{\Bbb F_q}\overline{\Bbb F_q})/K$, of Galois group $\Gamma\simeq \widehat{\Bbb Z}$, whose natural topologial generator is 
 the $q$--power Frobenius automorphism of $\overline{\Bbb F_q}$, denoted by $\gamma$. 

 \begin{theorem}[Greither--Popescu \cite{GP12}]\label{GP-main-intro}
   For $(K/k, S, \Sigma)$ as above and all primes $\ell$ we have
\begin{enumerate}
\item ${\rm pd}_{\Bbb Z_\ell[[G\times\Gamma]]}(T_\ell(M_{S,\Sigma}))=1.$
\item ${\rm Fitt}_{\Bbb Z_\ell[[G\times\Gamma]]}(T_\ell(M_{S,\Sigma}))=\langle \Theta_{S, \Sigma}(\gamma^{-1})\rangle.$
\end{enumerate}
\end{theorem}

In the statement above, ${\rm pd}_R(M)$ and ${\rm Fitt}_R(M)$ denote the projective dimension,
respectively the $0$--th Fitting ideal of a finitely presented module $M$ over a commutative, unital ring $R$.
See \S4 in \cite{GP12} for the relevant definitions and properties of Fitting ideals.

\subsection{Geometric Iwasawa Theory} The main goal of this paper is to prove analogues of Theorem \ref{GP-main-intro}
above along geometric Iwasawa--towers $K_\infty/K$, which are highly ramified and obtained from $K$ essentially by adjoining  the
$\frak p^n$--torsion points of a sign--normalized, rank one Drinfeld module (a Hayes module),
for some place $\frak p$ in $k$ and all $n\in\Bbb Z_{\geq 1}$.

This geometric Iwasawa theoretic approach was first considered in \cite{Angles},
in the particular case where $K=k=\Bbb F_q(t)$ and $K_\infty=\cup_{n\geq 0}  K_n$ with $K_n$ obtained by adjoining the
$\frak p^{n+1}$--torsion $\mathcal C[\frak p^{n+1}]$ of the Carlitz module 
\[
  \mathcal C: \Bbb F_q[t]\to \Bbb F_q[t]\{\tau\}, \qquad \mathcal C(t)=t+\tau,
\]
for a maximal ideal $\frak p$ in $\Bbb F_q[t]$. The fields $K_n$ are the ray--class fields of $k$ of
conductors $(\frak p^{n+1}v_\infty)$, where $v_\infty$ is the valuation of $k$ of uniformizer $1/t$. While using
Theorem \ref{GP-main-intro} above and the techniques and results developed in \cite{GP12}, the authors of \cite{Angles} are
studying the more classical Iwasawa
$\Bbb Z_p[[G(K_\infty/k)]]$--module
$$\frak X_p^{(\infty)}:=\varprojlim_n\, ({\rm Pic}^0(K_n)\otimes\Bbb Z_p),$$
where the projective limit is taken with respect to the usual norm maps at the level of the Picard groups of the function fields $K_n$.
One has topological group isomorphisms
$$G(K_\infty/K)\simeq \Bbb F_{\frak p}^\times\times U_{\frak p}^{(1)}\simeq \Bbb F_{\frak p}^\times\times \Bbb Z_p^{\aleph_0},$$
where $\Bbb F_{\frak p}$ is the residue field of $\frak p$, $U_{\frak p}^{(1)}$ is the group of principal units in the
completion of $k$ at $\frak p$ and $\Bbb Z_p^{\aleph_0}$ denotes a product of countably many copies of $\Zp$.
The main Iwasawa theoretic result in \cite{Angles} gives the $0$--th Fitting ideal of $\frak X_p^\infty$, away from the trivial
character of $\Bbb F_{\frak p}^\times$, in terms of an element
$\Theta^{\infty, \sharp}_{S, \Sigma}\in \Bbb Z_p[[G(K_\infty/K)]]$, which should be viewed as the $\Bbb Z_p[[G(K_\infty/k)]]$--analogue of the special value $\Theta_{S, \Sigma}(1)\in\Bbb Z[G(K/k)]$ of the element $\Theta_{S,\Sigma}(u)$ described above.
The work in \cite{Angles} was further developed in \cite{Bandini} and \cite{Coscelli}, see Remark \ref{remark-Bandini-Coscelli}.
\\

As opposed to \cite{Angles}, the set-up of this paper is the following. We fix an arbitrary function field $k$ of exact field of constants $\Bbb F_q$  and a place $v_\infty$ of $k$, called the infinite place of $k$ from now on. We let $A$ denote the Dedekind domain consisting of those elements in $k$ which are integral at all places of $k$, except for $v_\infty$. Further, we fix an ideal $\frak f$ and a maximal ideal $\frak p$ of $A$, such that $\frak p\nmid\frak f$. The geometric extensions of $k$ of interest to us are the fields
$$L_n:=H_{\frak f\frak p^{n+1}}, \text{ for all }n\geq 0,$$
which are the ray--class fields of $k$ of conductors $\frak f\frak p^{n+1}$ in which $v_\infty$ splits completely (i.e. the {\it{real}} ray--class fields of conductors $\frak f\frak p^{n+1}$.) As proved by Hayes in \cite{Hay85}, the extension $L_n/L_0$ is essentially generated by the $\frak p^{n+1}$--torsion points of a certain type of rank 1, sign-normalized Drinfeld module defined on $A$. (See Section~\ref{Drinfeld modules and cft} for details.) The ensuing geometric Iwasawa tower $L_\infty/k$, with $L_\infty=\cup_nL_n$, has Galois group $G_\infty$ which sits in an exact sequence
$$0\to G(L_\infty/L_0)\to G_\infty\to G(L_0/k)\to 0,$$
where $G(L_\infty/L_0)\simeq \Bbb Z_p^{\aleph_0}$  and $G(L_0/k)$ finite. Since the ramification locus of $L_\infty/k$ is finite, namely $S:=\{\frak p\}\cup\{v\,\vert\, v \text{ prime in } A, v\vert\frak f\}$, one can construct the following element 
$$\Theta_{S, \Sigma}^{(\infty)}(u):=\varprojlim_n\Theta_{S,\Sigma}^{(n)}(u)\in\Bbb Z_p[[G_\infty]][[u]],$$
out of the polynomials $\Theta_{S,\Sigma}^{(n)}(u)\in\Bbb Z[G(L_n/k)][u]$ associated in \cite[\S4.2]{GP12} to the data $(L_n/k,S, \Sigma)$,
for any finite, non-empty set  $\Sigma$ of primes in $k$, disjoint from $S$. On the other hand,
to the set of data $(L_\infty/k, S, \Sigma)$ one can associate the following $\Zp[[G_\infty\times\Gamma]]$--module
$$T_p(M_{S,\Sigma}^{(\infty)}):=\varprojlim_n T_p(M_{S,\Sigma}^{(n)}),$$
where $M_{S,\Sigma}^{(n)}$ is the Picard $1$--motive for $(L_n/k, S, \Sigma)$, and $T_p(M_{S,\Sigma}^{(n)})$ is its $p$--adic Tate module, as defined in
\cite[ \S2]{GP12}. The projective limit is taken with respect to certain canonical norm maps, described in detail in \S3 below. It turns out that neither $T_p(M_{S, \Sigma}^{(n)})$ nor 
$T_p(M_{S, \Sigma}^{(\infty)})$ depend on $\Sigma$, reason for which we will drop $\Sigma$ from those notations.
In \S3.2, we prove  the following geometric--arithmetic analogue of Theorem \ref{GP-main-intro} above.
\begin{theorem}\label{EMC-I-intro}
  For any finite, non-empty set  $\Sigma$ of primes in $k$, disjoint from $S$, \\
  the $\Bbb Z_p[[G_\infty\times\Gamma]]$--module $T_p(M_{S}^{(\infty)})$ is finitely generated, torsion and
\begin{enumerate}
\item  ${\rm pd}_{\Zp[[G_\infty\times\Gamma]]}(T_p(M^{(\infty)}_{S}))=1.$
\item ${\rm Fitt}_{\Zp[[G_\infty\times\Gamma]]}(T_p(M^{(\infty)}_{S}))=\langle \Theta_{S,\Sigma}^{(\infty)}(\gamma^{-1})\rangle.$
\end{enumerate}
\end{theorem}
In order to obtain a geometric (along the tower $L_\infty/k$) Iwasawa main conjecture--type result, one has to take $\Gamma$--coninvariants. In \S3.3 we establish a $\Bbb Z_p[[G_\infty]]$--module isomorphism
$$T_p(M_S^{(\infty)})_{\Gamma}\simeq \Nabla_S^{(\infty)},$$
where $\Nabla_S^{(\infty)}$ is an arithmetically meaningful $\Bbb Z_p[[G_\infty]]$--module, a projective limit of Ritter--Weiss type modules $\nabla_S^{(n)}$ which are extensions of divisor groups by class groups. (See the Appendix. Also, see \cite{Ritter-Weiss} for the number field analogues of $\Nabla_S^{(n)}$.) We prove the following.
\begin{theorem}\label{EMC-II-intro}
  For any finite, non-empty set  $\Sigma$ of primes in $k$, disjoint from $S$, \\
the $\Bbb Z_p[[G_\infty]]$--module $\Nabla_S^{(\infty)}$ is finitely generated, torsion and
\begin{enumerate}
\item  ${\rm pd}_{\Bbb Z_p[[G_\infty]]}(\Nabla_S^{(\infty)})=1.$
\item ${\rm Fitt}_{\Zp[[G_\infty]]}(\Nabla_S^{(\infty)})=\langle \Theta_{S,\Sigma}^{(\infty)}(1)\rangle.$
\end{enumerate}
\end{theorem}

To relate our results to those in \cite{Angles, Bandini,Coscelli} we have to introduce some further notation.
  We let $\Delta$ denote the maximal subgroup of $G(L_0/k)$ whose order is not divisible by $p$. Then we have a canonical isomorphism
  $G_\infty \simeq \Delta \times G_\infty^{(p)}$ where $G_\infty^{(p)}$ is the maximal pro-$p$ subgroup of $G_\infty$. We view the idempotent
  $e_\Delta := \frac{1}{|\Delta|} \sum_{\delta \in \Delta}\delta$ as an element of $\Zp[[G_\infty]]$ and consider the exact functor
  $M \mapsto M^\sharp := (1 - e_\Delta)M$ from the category of $\Zp[[G_\infty]]$-modules to the category of modules over the
  quotient ring $\Zp[[G_\infty]]^\sharp:=(1-e_{\Delta})\Bbb Z_p[[G_\infty]]$.

  Further, if $\frak f$ is the unit ideal $\frak e$ and the prime $\frak p$ stays inert in the real Hilbert class field $H_{\frak e}$ over $k$,
  then for $S = \{\frp\}$
  one has an isomorphism of $\Bbb Z_p[[G_\infty]]^\sharp$--modules (see \S3.3 below)
  \[
    \Nabla_S^{(\infty), \sharp}\simeq \frak X_{p}^{(\infty), \sharp}.
  \]
  Since these additional hypotheses are obviously satisfied when $k=\Bbb F_q(t)$ (as $H_{\frak e}=k$ in that case),
  Theorem \ref{EMC-II-intro} above implies  the main 
  Iwasawa theoretic result in \cite{Angles} discussed above, for the real Carlitz tower, i.e the maximal subfield of
  $K_\infty$ where $v_\infty$ splits completely. (See Theorem \ref{limit theorem 3 sharp}.)
Under slightly stronger hypotheses, we obtain an isomorphism of $\Bbb Z_p[[G_\infty]]$--modules
$\Nabla_S^{(\infty)}\simeq \frak X_{p}^{(\infty)}$
which leads to a full description of ${\rm Fitt}_{\Bbb Z_p[[G_\infty]]}(\frak X_{ p}^{(\infty)})$. (See Theorem \ref{limit theorem 3}.)
For an even more detailed comparison of our results with those in \cite{Angles, Bandini,Coscelli} we refer the reader to Remark~\ref{remark-Bandini-Coscelli} below.
\smallskip

In order to establish the link between the modules $\Nabla_S^{(\infty)}$ and $\frak X_{p}^{(\infty)}$ mentioned above,
we needed to provide slight generalizations of the results in 
\cite{GPff} on Ritter-Weiss modules and Tate sequences for function fields. This is done in the Appendix.\\

\begin{remark} 
Unlike in classical Iwasawa theory, all Iwasawa algebras $\Bbb Z_p[[G_\infty\times\Gamma]]$, $\Bbb Z_p[[\Gamma]]$ and $\Bbb Z_p[[G_\infty]]$
relevant in this context are not Noetherian. In particular, one has an isomorphism
$$\Bbb Z_p[[G_\infty]]\simeq \Bbb Z_p[t(G_\infty)][[X_1, X_2, \dots]]$$
of topological rings, where the power series ring has countably many variables and $t(G_\infty)$ is the (finite) torsion subgroup of $G_\infty$. Throughout, if $R$ is a commutative ring, the ring of power series in countably many variables with coefficients in $R$ is defined by
$$R[[X_1, X_2, \dots]]:=\varprojlim_nR[[X_1, X_2, \dots, X_n]],$$
where the transition maps $R[[X_1, X_2, \dots, X_{n+1}]]\to R[[X_1, X_2, \dots, X_{n}]]$ are the $R$--algebra morphisms sending $X_i\mapsto X_i$, for all $i\leq n$, and $X_{n+1}\mapsto 0$. 
\end{remark}

\section{Class--field theory and geometric Iwasawa towers}

\subsection{Sign normalized Drinfeld modules and class field theory}\label{Drinfeld modules and cft}

This subsection follows the exposition of Hayes \cite{Hay85}. In particular, we recall without proof the results of \S 4 of loc.cit.

Let $k$ be a global function field and let $\Fq$ be its field of constants. For a place (discrete, rank $1$ valuation) $v$ of $k$, we let $k_v$ denote the completion of $k$ in the $v$--adic topology. We let $\mathcal O_{k_v}$, $\frm_{k_v}$, $U_{k_v}$ be the ring of integers in $k_v$, its maximal ideal, and its group of units, respectively. As usual, we let $U_{k_v}^{(n)}:=1+\frm_{k_v}^n$, for all $n\geq 1$.
We denote by $d_v$ the degree of $v$ relative to $\Bbb F_q$ and by $\Bbb F_v$ its residue field.
By definition, we have $\Bbb F_v=\Bbb F_{q^{d_v}}$. 
Further, we fix a uniformiser $\pi_v\in k$, for all $v$ as above. For every $v$, we have group isomorphisms 
$$k_v^\times\simeq \pi_v^{\Bbb Z}\times U_{k_v}, \qquad U_{k_v}\simeq \Bbb F_v^\times\times U_{k_v}^{(1)}.$$

Now, we fix once and for all a place $v_\infty$ of $k$ (called the place at infinity) and let $A$ be
the Dedekind ring of elements of $k$ that are integral outside $v_\infty$. (In \cite{Hay85} our $A$ is denoted by $A_\infty$.) Note that the places $v$ of $k$, which are different from $v_\infty$, are in one--to--one correspondence with the maximal ideals of $A$.
For such a maximal ideal $\frak p$, we denote by $v_{\frak p}$ the corresponding place of $k$, viewed as a rank one, discrete valuation of $k$, normalized so that $v_{\frak p}(k^\times)=\Bbb Z$.

In order to simplify notations, we let $k_\infty=k_{v_\infty}$, $\pi_\infty:=\pi_{v_\infty}$, $\Bbb F_\infty:=\Bbb F_{v_\infty}$, $d_\infty:=d_{v_\infty}$, etc. The same notation principle applies to the finite places $v_\frak p$, namely
$k_{\frak p}:=k_{v_{\frak p}}$, etc. Finally, for any place $v$, we let ${\rm ord}_v: k_v^\times\to \Bbb Z$ be its associated valuation, normalized so that ${\rm ord}_v(\pi_v)=1$. \\

Let $I_A$ denote the group of fractional ideals of $A$ and let $P_A$ be the subgroup of principal ideals. Then
$\Pic(A) = I_A / P_A$ and we write $h_A := |\Pic(A)|$ for the class number of $A$. We write $D_k$ for the group of divisors
of $k$ and $D_k^0$ for the subgroup of divisors of degree zero. Let $\div \colon k^\times \lra D_k^0$ be the
divisor map and set $\Pic^0(k) := D_k^0/\div(k^\times)$. We write $h_k := |\Pic^0(k)|$. Recall that we have an exact sequence
$$0\to {\rm Pic}^0(k)\to{\rm Pic}(A)\overset{\widehat{\rm deg}}\longrightarrow \Bbb Z/d_\infty\to 0,$$
where $\widehat{\rm deg}$ is the degree modulo $d_\infty$ map.  Consequently, we have an equality $h_A = h_kd_\infty$.

\begin{definition} As in \cite{Hay85}, we define the following.
\begin{enumerate}
\item A finite, Galois extension $K/k$ is called real (relative to $v_\infty$) if $v_\infty$ splits completely in $K/k$, or, equivalently, if there exists a $k$--embedding $K\hookrightarrow k_\infty$.
\item For an integral ideal $\frak m\subseteq A$, $H_{\frak m}$ denotes the real ray--class field of $k$ of conductor $\frak m$.
\item If $\frak m=\frak e:=A$ is the unit ideal, then we call $H_{\frak e}$ the real Hilbert class field of $k$.
\end{enumerate}
\end{definition}

Next, we give the id\`ele theoretic description of the class fields $H_\frm$, as in \cite{Hay85}. For that, let $J_k$ denote the group of id\`eles of $k$ and consider the following subgroups of $J_k$.
$$U(\frm):= k_\infty^\times\times\prod_{\frak p\mid \frak m}U_{\frak p}^{(v_{\frak p}(\frak m))}\times\prod_{\frak p\nmid\frak m\infty}U_{\frak p},\qquad J_{\frak m}:=k^\times\cdot U(\frak m).$$
The following is proved in \cite{Hay85}. 
\begin{prop}\label{cft-hm-prop} For all $\frak m$ as above, the Artin reciprocity map gives group isomorphisms
$$
J_k/J_{\frak m}\simeq G(H_{\frak m}/k), \qquad J_k/J_{\frak e}\simeq G(H_{\frak e}/H), \qquad 
J_{\frak e}/J_{\frak m}\simeq G(H_{\frak m}/H_{\frak e}).
$$
Further, if $\frak m\ne \frak e$, we have canonical group isomorphisms 
$$ J_k/J_{\frak e}\simeq {\rm Pic}(A), \qquad J_{\frak e}/J_{\frak m}\simeq (A/\frak m)^\times/\Bbb F_q^\times.$$
\end{prop}

The real ray--class fields $H_{\frak m}$ are contained in slightly larger abelian extensions $H_{\frak m}^\ast/k$, of conductor $\frak m\cdot v_\infty$, tamely ramified at $v_\infty$. The advantage of passing 
from $H_{\frak m}$ to $H_{\frak m}^\ast$ is that the latter can be explicitly constructed by adjoining torsion points of certain rank $1$ $A$--Drinfeld modules to the field of definition $H_{\frak e}^\ast$ of these Drinfeld modules. 
Next, we give the id\`ele theoretic description and explicit construction of the fields $H_{\frak m}^\ast$, both due to Hayes \cite{Hay85}.\\

As in \cite{Hay85}, let us fix a sign function 
$${\rm sgn}:  k_\infty^\times\to \Bbb F_\infty^\times.$$
By definition, this is a group morphism, such that ${\rm sgn}(U_{\infty}^{(1)})=1$ and ${\rm sgn}\vert_{\Bbb F_\infty^\times}={\rm id}_{\Bbb F_\infty^\times}.$  Note that ${\rm sgn}$ is uniquely determined by the value
${\rm sgn}(\pi_\infty)$ at the fixed uniformiser $\pi_\infty$.\\

For every integral ideal $\frak m\subseteq A$, we define the following subgroups of $J_k$.
$$U^*(\frm) := \{\left(\alpha_v\right)_v \in U(\frak m) \mid\, {\rm sgn}(\alpha_\infty)=1\},\quad J_{\frak m}^\ast:=k^\times\cdot U^*(\frak m).$$
\begin{definition}
For all integral ideals $\frm \sseq A$, we define $H_{\frak m}^\ast$ to be the unique abelian extension of $k$ which corresponds to the subgroup $J_{\frak m}^\ast$ of $J_k$ via the standard class--field theoretic correspondence. 
\end{definition}

For all $\frak m$ as above, with $\frak m\ne \frak e$, we have a canonical commutative diagram of group morphisms with exact rows and columns.

\begin{equation*}
\xymatrix{0\ar[r] & \Bbb F_q^\times\ar[r]\ar@{^{(}->}[d] & \Bbb F_\infty^\times\ar[r]\ar@{^{(}->}[d] & \Bbb F_\infty^\times/\Bbb F_q^\times\ar[r]\ar@{^{(}->}[d] &0\\
0\ar[r] & (A/\frak m)^\times \ar[r]\ar@{>>}[d] & J_k/J_{\frak m}^\ast\ar[r]\ar@{>>}[d] & J_k/J_{\frak e}^\ast\ar[r]\ar@{>>}[d] &0\\
0\ar[r] & (A/\frak m)^\times/\Bbb F_q^\times \ar[r] & J_k/J_{\frak m}\ar[r] & J_k/J_{\frak e}\ar[r] &0}
\end{equation*}\\\\
As a consequence, we have the following diagram of abelian extensions of $k$, whose relative Galois groups are canonically isomorphic to the labels on the connecting line segments.

\begin{equation}\label{field-diagram-1}
\xymatrix{
& & H_\frm^\ast \\
& H_\fre^*H_\frm \ar@{-}[ur]^{\Bbb F_q^\times} & \\
H_\fre^* \ar@{.}@/^2.0pc/[uurr]^{(A/\frm)^\times} \ar@{-}[ur] \ar@{-}[dd]_{\Bbb F_\infty^\times/\Bbb F_q^\times} & & \\
& H_\frm \ar@{-}[uu]\ar@{.}[uuur]_{\Bbb F_\infty^\times}& \\
H_\fre \ar@{-}[ur]_{\frac{(A/\frak m)^\times}{\Bbb F_q^\times}} & & \\
k \ar@{-}[u]^{{\rm Pic}(A)} &&
}
\end{equation}\\
The extensions $H_\fre^\ast/H_\fre$ and $H_\frm^\ast/H_\frm$ are totally and tamely ramified at the primes above $v_\infty$.\\

Next, we describe the explicit Drinfeld modular construction of the fields $H_\fre^\ast$ and $H_\frm^\ast$, for all $\frm\ne \fre$. Let $\Bbb C_\infty$ be the $v_\infty$--completion of the algebraic closure of $k_\infty$. Let $\Bbb C_\infty\{\tau\}$ be the non-commutative ring of twisted polynomials with the rule $\tau\omega = \omega^q\tau$, for
$\omega \in \Bbb C_\infty$. We write
\[
D \colon  \Bbb C_\infty\{\tau\} \lra \Bbb C_\infty, \quad a_0\tau^0 + a_1\tau^1 + \ldots + a_d\tau^d \mapsto a_0,
\]
for the constant term map.

\begin{definition}[Hayes] A map
 $\rho \colon A \lra \Bbb C_\infty\{\tau\}$, $x \mapsto \rho_x$, is called a sgn--normalized Drinfeld module of rank one if the following are satisfied.
\begin{itemize}
\item [(a)] $\rho$ is an $\Fq$-algebra homomorphism.
\item [(b)] $\deg_\tau(\rho_x) = \deg(x):={\rm dim}_{\Bbb F_q}(A/xA)=-{\rm ord}_{v_\infty}(x)d_\infty$, for all $x\in A$.
\item [(c)] The map $A \lra \Bbb C_\infty$, $x \mapsto D(\rho_x)$, is the inclusion $A \sseq \Bbb C_\infty$.
\item [(d)] If $s_\rho(x)$ denotes the leading coefficient of $\rho_x \in \Bbb C_\infty\{\tau\}$, then $s_\rho(x) \in \Finfty^\times$, for
all $x \in A$. 
\item[(e)] If one extends $s_\rho$ to $k_\infty^\times$ by sending 
\[
x = \sum_{i=i_0}^\infty a_i \pi_\infty^i  \mapsto a_{i_0}s_\rho(\pi_\infty)^{i_0}, 
\]
if $a_{i_0}\ne 0$, then $s_\rho \colon k_\infty^\times \lra \Finfty^\times$ is a twist of ${\rm sgn}$, i.e there exists $\sigma \in \Gal(\Finfty/\Fq)$, such
that $s_\rho = \sigma \circ \sgn$.
\end{itemize}
\end{definition}
Any Drinfeld module as above endows $(\Bbb C_\infty, +)$ with an $A$--module structure given by 
$$x\ast z:=\rho_x(z), \qquad \text{ for all } x\in A,\,  z\in\Bbb C_\infty,$$
where $(\sum_i a_i\tau^i)(z)=\sum_i a_iz^{q^i}$.

\begin{definition} Let $\rho:A\to\Bbb C_\infty\{\tau\}$ be a rank $1$, sgn--normalized Drinfeld module as above.
\begin{itemize}
\item[(a)] The minimal field of definition $k_\rho$ of $\rho$ is the extension of $k$ inside $\Bbb C_\infty$ generated by the coefficients of the twisted polynomials $\rho_x$, for all $x\in A$.
\item[(b)] For all integral ideals $\frm\subseteq A$ with $\frm\ne\fre$, we let
$$\rho[\frm] := \{ \alpha \in \Bbb C_\infty\mid \rho_x(\alpha) = 0 \text{ for all } x \in \frm\}$$
denote the $A$--module of $\frm$--torsion points of  $\rho$.
\end{itemize}
\end{definition}

The following gives an explicit construction of the class--fields $H^\ast_\frm$. (See \cite[\S4]{Hay85} for proofs.)
\begin{prop}[Hayes \cite{Hay85}] Let $\rho$ be a rank $1$, sgn--normalized Drinfeld module as above. Then,  the following hold, for all ideals $\frm\subseteq A$, with $\frm\ne\fre$.
\begin{enumerate}
\item The minimal field of definition $k_\rho$ of $\rho$ equals $H_\fre^\ast$.
\item We have an equality $H_\frm^\ast=H_\fre^\ast(\rho[\frm])$.
\item The $A/\frm$--module $\rho[\frm]$ is free of rank $1$ and, via the canonical isomorphism 
$$(A/\frm)^\times\simeq G(H_\frm^\ast/H_\fre^\ast), \qquad \widehat x\to\sigma_{\widehat x},$$
we have $\sigma_{\widehat x}(\alpha)=s_\rho(x)^{-1}\cdot\rho_x(\alpha)$, for all $x\in A$ coprime to $\frm$ and all $\alpha\in\rho[\frm]$.
\end{enumerate}
\end{prop}

\begin{remark}
The proposition above should be viewed as the function field analogue of the theory of complex multiplication for quadratic imaginary fields, where the role of $\rho$ is played by an elliptic curve with CM by the ring of integers
of a quadratic imaginary field $k$.
\end{remark}

\subsection{$\Zp^{\aleph_0}$--extensions (Geometric Iwasawa towers)}\label{section-geometric-Iwasawa-towers}

We continue to use the notation of Subsection \ref{Drinfeld modules and cft}.  We fix a prime ideal $\frp$ of $A$ and an integral ideal $\frf \nt A$
which is coprime to $\frp$.  For all $n \ge 0$, we consider the following abelian extensions of $k$, viewed as subfields of $\Bbb C_\infty$:
\[
L_n := H_{\frf\frp^{n+1}}, \quad L_n^\ast := H^\ast_{\frf\frp^{n+1}}
\]
and set
\[
L_\infty := \bigcup_{n \ge 0} L_n, \quad L_\infty^\ast := \bigcup_{n \ge 0} L_n^\ast.
\]
For all $n\geq 0$, we let $G_n:=G(L_n/k)$, $\Gamma_n:=G(L_n/L_0)$. Also, we let $G_\infty:=G(L_\infty/k)$, $\Gamma_\infty:=G(L_\infty/L_0)$. The results in the previous section show that we have the following commutative diagrams of abelian groups with exact rows and canonical vertical isomorphisms.\\
\begin{equation*}
\xymatrix{
0 \ar[r] & U_{k_\frp}^{(1)}/U_{k_\frp}^{(n+1)}  \ar[r] \ar[d]^\wr  & (A/\frf\frp^{n+1})^\times/\Bbb F_q^\times \ar[r] \ar[d]^\wr &(A/\frf\frp)^\times/\Bbb F_q^\times \ar[r] \ar[d]^\wr &  0 \\
0 \ar[r] & \Gamma_n  \ar[r]   & G(L_n/H_{\frak e})\ar[r] & \Gal(L_0/H_\fre)\ar[r]  &  0
}
\end{equation*}\\
Consequently, for all $n\geq 0$, we obtain the following diagram of field extensions whose relative
Galois groups are canonically isomorphic to the labels of the connecting line segments.\\
\begin{equation}\label{field-diagram-2}
\xymatrix{
& & H^\ast_{\frf\frp^{n+1}}  \\
& H_{\frf\frp}^\ast \ar@{-}[ur]^{U_{k_\frp}^{(1)}/U_{k_\frp}^{(n+1)}}  & \\
& H_\fre^* H_{\frf\frp} \ar@{-}[u]^{\Bbb F_q^\times} & H_{\frf\frp^{n+1}} \ar@{-}[uu]_{\Bbb F_\infty^\times} \\
H_\fre^* \ar@{-}[ur]  & H_{\frf\frp} \ar@{-}[ur]_{U_{k_\frp}^{(1)}/U_{k_\frp}^{(n+1)}} \ar@{-}[u] & \\
H_\fre \ar@{-}[ur]_{\frac{(A/\frf\frp)^\times}{\Bbb F_q^\times}} \ar@{-}[u]^{\Bbb F_\infty^\times/\Bbb F_q^\times} & & \\
k \ar@{-}[u]^{{\rm Pic}(A)} \ar@{-}@{.}@/_3pc/[uuurr]_{G_n} \ar@{-}@{.}@/_2pc/[uur]_{G_0}& &
}
\end{equation}\\
Further, we obtain topological group isomorphisms
\[
\Gamma_\infty := \Gal(L_\infty/L_0) \simeq \Gal(L_\infty^\ast/L_0^\ast) \simeq \varprojlim_n \frac{U_{k_\frp}^{(1)}}{U_{k_\frp}^{(n+1)}} \simeq U_{k_\frp}^{(1)}.
\]
Now, we recall the following structure theorem, due to Iwasawa. (See also \cite[Satz II.5.7]{Neu99}).

\begin{theorem}[{Iwasawa \cite{Iw86}}]
  Let $K$ be a local field of characteristic $p > 0$ and let $U_K^{(1)}$ denote its group of principal units. Then,
there is an isomorphism of topological groups
\[
U_K^{(1)}\simeq \Zp^{\aleph_0},
\]
where the right side denotes a direct product of countably many copies of $(\Bbb Z_p, +)$, endowed with the product of the $p$--adic topologies.
\end{theorem}
As a consequence, we have an isomorphism of topological groups
\begin{equation}\label{Gamma-infinity-isomorphism}\Gamma_\infty:=G(L_\infty/L_0)\simeq \Bbb Z_p^{\aleph_0}.\end{equation}\\
The following gives a description of the Iwasawa algebras relevant in our considerations below.
\begin{prop}\label{big-Iwasawa-algebra-prop}
Let $(\mathcal O, \frak m_{\mathcal O})$ be a local, compact $\Bbb Z_p$--algebra, which is $\frak m_{\mathcal O}$--adically complete. If $\mathcal G$ is an abelian pro--$p$ group, topologically isomorphic to $\Bbb Z_p^{\aleph_0}$, then the following hold.
\begin{enumerate}
\item 
There is an isomorphism of topological $\mathcal O$--algebras
$$\mathcal O[[\mathcal G]]\simeq \mathcal O[[X_1, X_2, \dots]],$$
where the left side is endowed with the profinite limit topology and the right side with the projective limit of the $(\frak m_{\mathcal O}, X_1, \dots, X_n)$--adic topologies on each $\mathcal O[[X_1, \dots, X_n]]$, as $n\to\infty$. 
\item If $\mathcal O$ is an integral domain, then $\mathcal O[[\mathcal G]]$ is a local, integral domain.
\item If $\mathcal O$ is a PID, then $\mathcal O[[\mathcal G]]$ is a UFD and, therefore, normal. 
\end{enumerate}
\end{prop}
\begin{proof}{\it (Sketch.)}
(1) Use induction on $n$ and the Weierstrass Preparation Theorem (see Thm. 2.1 in \cite[Ch.5, \S2]{Lang-cyclo}) to show that one has an isomorphism of topological $\mathcal O$--algebras
$$\mathcal O[[\Bbb Z_p^n]]\simeq \mathcal O[[X_1, \dots, X_n]].$$
Then, pass to a projective limit with respect to $n$ to get the desired isomorphism.

(2) This is Lemma 1 in \cite{Nishimura-I}. Note that with the notations and definitions of loc.cit. we have $\mathcal O[[X_1, X_2, \dots]]=\mathcal O\{X\}_{\aleph_0}$, where $X$ is a set of cardinality $\aleph_0$.

(3) This is Theorem 1 in \cite{Nishimura-I}. See the note above regarding the notations in loc.cit.
\end{proof}

\begin{remark} Typical examples of $\Bbb Z_p$--algebras $\mathcal O$ as in the Proposition above are rings of integers $\mathcal O_F$ in finite extensions $F/\Bbb Q_p$ of $\Bbb Q_p$.  Also, group rings  
$\mathcal O_F[P]$, where $P$ is a finite, abelian $p$--group satisfy the hypotheses of part (1), but not parts (2)--(3) of the proposition above. Note that in the latter case the maximal ideal of $\mathcal O_F[P]$ is given by  $\frak m_{\mathcal O_F[P]}=(\frak m_{\mathcal O_F}, I_P)$, where $I_P$ is the augmentation ideal of $\mathcal O_F[P]$. Note that if $P$ is a product of $r$ cyclic groups of orders $p^{n_1}, \dots, p^{n_r}$, respectively, then we have isomorpisms of topological $\mathcal O_F$--algebras
\begin{eqnarray*}
\mathcal O_F[P]&\simeq &\mathcal O_F[X_1, \dots, X_r]/\left((X_1+1)^{p^{n_1}}-1, \dots (X_r+1)^{p^{n_r}}-1\right)\\
 &\simeq& O_F[[X_1, \dots, X_r]]/\left((X_1+1)^{p^{n_1}}-1, \dots (X_r+1)^{p^{n_r}}-1\right),
\end{eqnarray*}
where the first isomorphism sends the generators of $P$ to $\widehat{X_1+1}, \dots, \widehat{X_r+1}$, respectively,
and the second is a consequence of the Weierstrass preparation theorem cited above, applied inductively.
Since the right-most algebra is clearly complete in its $(\frak m_{\mathcal O_F}, X_1, \dots, X_r)$--adic topology, the left-most algebra is also complete in its $\frak m_{{\mathcal O}_F[P]}$--adic topology.
\end{remark}\medskip

We end this section with a result on the decomposition groups $G_v(L_\infty/L_n)$ in the extension $L_\infty/L_n$, for all primes $v\vert\frak{f}$ and a fixed $n\geq 0$. This will be used in the proof of Proposition \ref{not a zero divisor prop} below. To that end, fix $n\geq 0$ and for every prime $v\vert\frak{f}$, let $U_{S_v}$ be the group of $S_v$--units in $k^\times$, where $S_v:=\{v, \infty\}$. We remind the reader that these are the elements of $k^\times$ whose divisor is supported at $S_v$.  Consequently, we have a group isomorphism
$$U_{S_v}\simeq\Bbb F_q^\times\times \Bbb Z.$$
Further we let $U_{S_v}^{(n+1)}:=\{x\in U_{S_v}\mid x\equiv 1 \mod (\frac{\frak{f}}{v^{{\rm ord}_v(\frak f)}}\cdot\frak p^{n+1})\}.$ This is a subgroup of finite index in $U_{S_v}$ which is torsion free. Therefore, it is infinite cyclic
$$U_{S_v}^{(n+1)}=x_v^{\Bbb Z},$$
generated by some $x_v\in k^\times$, which obviously satisfies the following
\begin{equation}\label{divisors}{\rm div}(x_v)={\rm ord}_v(x_v)\cdot v+{\rm ord}_{\infty}(x_v)\cdot\infty, \quad {\rm ord}_v(x_v)={-(d_\infty/d_v)}\cdot{\rm ord}_\infty(x_v)\ne 0.\end{equation}
\medskip

In what follows, we let $U(\frak f\frak p^\infty):=\bigcap_n U(\frak f\frak p^n)$ and let $i_v:k_v^\times\to J_k/k^\times U(\frak f\frak p^\infty)$ be the standard morphism
(sending $x\in k_v^\times$ into the class of the id\`ele having $x$ in the $v$--component and $1$ everywhere else),
for all primes $v$ of $k$. Now, we consider the topological group isomorphism 
 $$\rho_{\frak p}^{(n)}: U_{k_{\frak p}}^{(n+1)}\simeq G(L_\infty/L_n)$$
 obtained by composing the Artin reciprocity isomorphism $\rho \colon J_k/k^\times U({\frak f\frak p^\infty})\simeq G(L_\infty/k)$
 with the standard embedding  
 $U_{k_{\frak p}}^{(n+1)}\subseteq k_{\frak p}^\times\overset{i_{\frak p}}\longrightarrow J_k/k^\times U({\frak f\frak p^\infty})$. Note that $i_{\frak p}$ restricted to $U_{k_{\frak p}}^{(n+1)}$ is indeed injective, for all $n\geq 0$.
 
\begin{prop}\label{decomposition groups prop} With notations as above,  the following hold.
\begin{enumerate}
\item  For all primes $v\vert\frak f$, if we let $x_v^{\Bbb Z_p}$ denote the cyclic $\Bbb Z_p$--submodule of $U_{k_{\frak p}}^{(n+1)}$ generated by $x_v$, then $\rho_{\frak p}^{(n)}$ gives an isomorphism of topological groups
$$\rho_{\frak p}^{(n)}: x_v^{\Bbb Z_p}\simeq G_v(L_\infty/L_n).$$
\item Let $G_{\frak f}(L_\infty/L_n)$ be the subgroup of $G(L_\infty/L_n)$ generated by $G_v(L_\infty/L_n)$, for all $v\vert\frak f$. Then, if we let $f:={\rm card}\{v\mid v\vert\frak f\}$, we have topological group isomorphisms 
$$G_{\frak f}(L_\infty/L_n)\simeq \prod_{v\mid\frak f} G_{v}(L_\infty/L_n)\simeq \Bbb Z_p^f$$
\end{enumerate}
\end{prop}
\begin{proof}
  (1) A well--known class--field theoretical fact gives an equality of groups
$$G_v(L_\infty/L_n)=\overline{\rho(i_v(k_v^\times))\cap\rho(i_{\frak p}(U_{k_{\frak p}}^{(n+1)}))},$$
where $\overline{X}$ denotes the pro-$p$ completion (topological closure)
of the subgroup $X$ inside the pro-$p$ group $G(L_\infty/L_n)$. However, it is easily seen that we have 
$$\rho(i_v(k_v^\times))\cap\rho(i_{\frak p}(U_{k_{\frak p}}^{(n+1)}))=\rho_{\frak p}^{(n)}(x_v^{\Bbb Z}),$$
which, after taking the pro--$p$ completion of both sides, concludes the proof of part (1).\\

(2) According to part (1), it suffices to show that the elements $\{x_v\mid v\vert\frak f\}$ are $\Bbb Z_p$--linearly independent in $U_{k_{\frak p}}^{(n+1)}$. However, since their divisors are clearly  $\Bbb Z$--linearly independent (see \eqref{divisors} above), these elements are $\Bbb Z$--linearly independent in $k^\times$. Now, the function field (strong) analogue of Leopoldt's Conjecture, proved in \cite{Kisilevsky}, implies that the elements in question are $\Bbb Z_p$--linearly 
independent in  $U_{k_{\frak p}}^{(n+1)}$, as desired.
\end{proof}

\subsection{The basic example: The Carlitz module}
\label{Carlitz subsection}

We briefly describe the special situation which arises in the case of the Carlitz cyclotomic extension of a rational function field
(see e.g. \cite[Sec.~2]{Angles}).\\

Let $k = \Fq(\theta)$ be the rational function field over $\Bbb F_q$ and let $v_\infty$ correspond to the valuation on $k$ of uniformizer $1/\theta$. 
Then $A = \Fq[\theta]$. Furthermore, $h_k = 1$, $d_\infty = 1$, and $H_\fre^* = H_\fre = k$. We consider the Carlitz module
\[
\mathcal C \colon A \lra k\{\tau\}, \quad \theta \mapsto \mathcal C(\theta) = \theta\tau^0 + \tau^1,
\]
which is sgn--normalized with respect to the unique sign function satisfying ${\rm sgn}(1/\theta)=1.$ 
 All data in the following refers to $\rho = \mathcal C$.\\

For each $\frm \ne \fre$ we have 
\[
\Gal(H^\ast_\frm/k) \simeq \left( A/\frm \right)^\times, \quad \Gal(H_\frm/k) \simeq \left( A/\frm \right)^\times / \Fq^\times, \quad \Gal(H_\frm^\ast/H_\frm)\simeq \Bbb F_q^\times=\Bbb F_\infty^\times, 
\]
and the subgroup $\Fq^\times \hookrightarrow  \left( A/\frm \right)^\times$ identifies with the decomposition subgroup 
at $\infty$ which also
equals the ramification subgroup at $v_\infty$ for the extension $H_\frm^\ast/k$. \\

We fix a prime $\frp$ of $A$ of degree $d = d_\frp$ and consider the fields $L^\ast_n := H^\ast_{\frp^{n+1}}$ for $n \ge 0$.
Then
\begin{equation}\label{decomp of group algebra}
G_n^\ast := \Gal(L_n^\ast/k) = \Delta^\ast \times \Gamma_n \simeq \left( A/\frp^{n+1} \right)^\times \simeq 
\left( A/\frp \right)^\times \times \frac{U_{k_\frp}^{(1)}}{U_{k_\frp}^{(n+1)}},
\end{equation}
where $\Delta^\ast \simeq \Gal(L_0^\ast/k) \simeq \left( A/\frp \right)^\times$ is cyclic of order $q^d-1$ and 
$\Gamma_n \simeq  \Gal(L^\ast_n/L^\ast_0) \simeq U_{k_\frp}^{(1)}/U_{k_\frp}^{(n+1)}$ is the $p$-Sylow subgroup of $G^\ast_n$.\\

The extension $L^\ast_n/k$ is unramified outside $\{v_\infty, \frp\}$, totally ramified at $\frp$ and tamely ramified of ramification degree $(q-1)$ at $v_\infty$.
More precisely, the decomposition field at $v_\infty$ is $L_n := H_{\frp^{n+1}}$ and $L^\ast_n/L_n$ is totally ramified of degree $(q-1)$.\\

Hence, the extension $L^\ast_\infty/k$  is also unramified outside $\{v_\infty, \frp\}$, 
totally ramified at $\frp$ and tamely ramified of degree $(q-1)$ at $v_\infty$.
More precisely, the decomposition field at $v_\infty$ is $L_\infty := \cup_n H_{\frp^n}$ and $L^\ast_\infty/L_\infty$ 
is totally ramified at $v_\infty$ of degree $(q-1)$.  We have
\[
G_\infty^\ast := \Gal(L_\infty^\ast/k) = \Delta^\ast \times \Gamma_\infty \text, \quad \Gamma_\infty\simeq U_{k_\frp}^{(1)}.
\]
Here the isomorphism $\Gamma_\infty \simeq U_{k_\frp}^{(1)}$ is induced by the $\frp$-cyclotomic character 
\[
\kappa \colon G_\infty \lra U_{k_\frp}^{},
\]
which is defined as follows. We write $A_\frp$ for the completion of $A$ at $\frp$, so that $A_\frp$ identifies with
the valuation ring of $k_\frp$, in particular, $A_\frp^\times = U_{k_\frp}$. Then,  $\phi$ can be uniquely extended to a formal Drinfeld module
(see \cite{Rosen03})
\[
\widehat{\mathcal C} \colon A_\frp \lra A_\frp \{\{\tau \}\}.
\]
Then, for any $\sigma \in G_\infty$ the value $\kappa(\sigma)$ is determined by the equality
\[
\sigma(\epsilon) = \widehat{\mathcal C}_{\kappa(\sigma)}(\epsilon) \text{ for all } \epsilon \in \mathcal C[\frp^\infty].
\]

Finally, we note that
\[
G_\infty := \Gal(L_\infty/k) = \Delta \times \Gamma_\infty, 
\]
where $\Delta  := \Delta^\ast / \Fq^\times\simeq (A/\frp)^\times/\Bbb F_q^\times$.

\section{Equivariant main conjectures in positive characteristic}

\subsection{Review of the work of Greither and Popescu}
\label{GP review}

In what follows, if $G$ is a finite, abelian group and $F$ is a field of characteristic $0$,  we denote by $\widehat G(F)$ the set of equivalence classes of the $\overline F$--valued characters $\chi$ of $G$, with respect to the equivalence relation 
$\chi\sim\chi' $ if there exists $\sigma\in G(\overline F/F)$, such that $\chi'=\sigma\circ\chi$. 

If $R$ is a commutative ring and $M$ is a finitely presented $R$--module, we let ${\rm Fitt}_R(M)$ denote the $0$--th Fitting ideal of $M$. For the definitions and relevant properties 
of Fitting ideals needed in this context, the reader may consult \cite{GP12}.

We let $K/k$ denote an abelian extension of characteristic $p$ global fields, of Galois group $G$. We assume that $\Fq$ is the exact field of 
constants of $k$ (but not necessarily of $K$). Let $X \lra Y$ be the corresponding $G$-Galois cover of smooth projective curves defined over
$\Fq$. Let $S$ and $\Sigma$ be two finite,  non-empty, disjoint sets of closed points of $Y$, such that $S$ contains
the set $S_{\mathrm{ram}}$ of points which ramify in $X$. We let $\overline{\Bbb F_q}$ denote the algebraic closure of $\Fq$ and set
$\bar{X} := X \times_{\Fq} \overline{\Bbb F_q}$, $\bar{Y} := Y \times_{\Fq} \overline{\Bbb F_q}$. Also, $\bar{S}$ and $\bar{\Sigma}$ denote the set of points on $\bar{X}$ sitting
above points of $S$ and $\Sigma$, respectively. 

For every unramified closed point $v$ on $Y$ we denote by $G_v$ and $\sigma_v$ the decomposition group and the Frobenius automorphism
associated to $v$. As before, we write $d_v$ for the residual degree over $\Fq$ and we let $Nv := q^{d_v} = | \F_{q^{d_v}} |$ denote the 
cardinality of the residue field associated to $v$.

To the set of data $(K/k, \Fq, S, \Sigma)$, one can associate a polynomial equivariant $L$-function
\begin{equation}
\label{eq L poly}
\Theta_{S, \Sigma}(u) := \prod_{v \in \Sigma}\left( 1 - \sigma_v^{-1} \cdot (qu)^{d_v} \right) \cdot
\prod_{v \not\in S}\left( 1 - \sigma_v^{-1} \cdot u^{d_v} \right)^{-1}.
\end{equation}
The infinite product on the right is taken over all closed points in $Y$ which are not in $S$. This product converges in $\ZG[[u]]$ and in fact
it converges to an element in the polynomial ring  $\ZG[u]$. We recall the link between $\Theta_{S, \Sigma}(u)$ and classical Artin $L$-functions.
For every complex valued irreducible character $\chi$ of $G$ we let $L_{S, \Sigma}$ denote the $(S, \Sigma)$-modified Artin $L$-function
associated to $\chi$. This is the unique holomorphic function of the complex variable $s$ satisfying the equality
\begin{equation}\label{poly versus L fct}
L_{S, \Sigma}(\chi, s) =   \prod_{v \in \Sigma}\left( 1 - \chi(\sigma_v)Nv^{1-s} \right) \cdot
\prod_{v \not\in S}\left( 1 - \chi(\sigma_v) (Nv)^{-s} \right)^{-1} 
\end{equation}
for all  $s \in \Ce$ with  $\mathrm{Re}(s) > 1$. Then, for all $s \in \Ce$,
\begin{equation}\label{StickelArtin}
\Theta_{S, \Sigma}(q^{-s}) = \sum_{\chi \in \widehat{G}(\Bbb C)} L_{S, \Sigma}(\chi, s) e_{\chi^{-1}},
\end{equation}
where $e_\chi := 1/|G| \sum_{g \in G}\chi(g) g^{-1} \in \Ce[G]$ denotes the idempotent corresponding to $\chi$.

We denote by $M_{\bar{S}, \bar{\Sigma}}$ the Picard $1$-motive associated to the set of data $(\bar{X}, \overline\Fq, \bar{S}, \bar{\Sigma})$,
see \cite[Def.~2.3]{GP12} for the definition. For a prime number $\ell$ we consider the $\ell$-adic Tate module (or $\ell$-adic
realization ) $T_\ell(M_{\bar{S}, \bar{\Sigma}})$, see \cite[Def.~2.6]{GP12}, endowed with the usual $\Zl[[G\times\Gamma]]$-module structure,
where $\Gamma := G(\overline\Fq / \Fq)$. Recall that $\Gamma$ is isomorphic to the profinite completion $\widehat{\Bbb Z}$ of $\Bbb Z$ and has a natural topological generator $\gamma$ given by the $q$-power
arithmetic Frobenius automorphism.

The main result of Section 4 of \cite{GP12} is the following.

\begin{theorem}[Greither--Popescu] \label{GP Th.4.3} The following hold for all prime numbers $\ell$.
  \begin{itemize}
  \item [(1)] The $\ZlG$-module $T_\ell(M_{\bar{S}, \bar{\Sigma}})$ is projective.
  \item [(2)] We have an equality of $\Zl[[G\times\Gamma]]$-ideals
\[
\left( \Theta_{S, \Sigma}(\gamma^{-1}) \right)  =  \Fitt_{\Zl[[G\times \Gamma]]}(T_\ell(M_{\bar{S}, \bar{\Sigma}})).
\]
  \end{itemize}

\end{theorem}

\begin{remark}\label{p-no Sigma-remark}
  By \cite[Rem.~2.7]{GP12} we have $T_p(M_{\bar{S}, \bar{\Sigma}}) = T_p(M_{\bar{S}, \emptyset})$. This is in accordance with
the fact that the product of Euler factors
$$
 \prod_{v \in \Sigma}\left( 1 - \sigma_v^{-1} \cdot (q\gamma^{-1})^{d_v} \right) 
$$
is a unit in $\Zl[[G \times \Gamma]]$. Indeed, from
\[
\sum_{n=0}^{N-1} \left( \sigma_v^{-1} (qu)^{d_v} \right)^n = \frac{1 - \left( \sigma_v^{-1} (qu)^{d_v} \right)^N}
{1 - \left( \sigma_v^{-1} (qu)^{d_v} \right)}
\]
and $\lim\limits_{N\ra\infty}\left(1 - \left( \sigma_v^{-1} (qu)^{d_v} \right)^N \right) = 1$ in $\Bbb Z_p[G][[u]]$, we see that
\[
\frac{1}{1 - \left( \sigma_v^{-1} (qu)^{d_v} \right)} = \sum_{n=0}^{\infty} \left( \sigma_v^{-1} (qu)^{d_v} \right)^n \in \Zp[G][[u]].
\]
\end{remark}

\subsection{The main results (Geometric Equivariant Main Conjectures)}

We fix a prime ideal $\frp$ and an integral ideal $\frf$ of $A$, such that $\frp \nmid \frf$. We will consider the tower of fields $H_{\frf\frp^{n+1}}/k$, for $n \ge 0$. (See \S2.2 and field diagram \eqref{field-diagram-2}.) 
The definition of the real ray--class fields $H_{\frf\frp^{n}}$ implies that we have
$$H_{\frf\frp^{n}} \cap \bar\F_q = \F_{q^{d_\infty}}, \qquad \bar\F_q H_\frf \cap H_{\frf\frp^{n}} = H_\frf,$$
for all $n\geq 0$. Consequently, we have the following perfect field diagram.

\begin{equation}\label{field-diagram-3}
\xymatrix{
&&& \calL_n := \bar\F_q H_{\frf\frp^{n+1}} \ar@{-}[dd] \ar@{-}[ddll]\\
&&&\\
& L_n := H_{\frf\frp^{n+1}} \ar@{-}[dd]  \ar@{-}@{.}@/_2pc/[ddddl]_{G_n} & & \calL_0 := \bar\F_q H_{\frf\frp}  \ar@{-}[d] \ar@{-}[ddll] \\
&&& \calK := \bar\F_q k \ar@{-}[d] \ar@{-}[ddll] \\
& L_0 := H_{\frf\frp} \ar@{-}[d]  \ar@{-}@{.}@/_1pc/[ddl]_{G_0} & & \kappa := \bar\F_q  \ar@{-}[ddll] \ar@{-}@{.}@/^1pc/[dddlll]^{\Gamma}\\
& E := L_0 \cap \calK \ar@{-}[d] \ar@{-}[dl] & & \\
k \ar@{-}[d] & \F_{q^{d_\infty}} \ar@{-}[dl] && \\
\Fq  &&&
}
\end{equation}

As in Subsection \ref{GP review} we let $S$ and $\Sigma$ be two finite, non-empty, disjoint sets of closed points of the smooth, projective curve
$Y$ corresponding to $k$, such that $S$ contains the set $S_{\mathrm{ram}}$ of points which ramify in $H_{\frf\frp^{n+1}}$. 
Note that this condition does not depend on $n$.

We write $\Theta_{S, \Sigma}^{(n)}(u) \in \Ze[G_n][u]$ for the equivariant $L$-function attached to 
$(L_n/k, \Fq, S, \Sigma)$
in (\ref{eq L poly}). We let $L_\infty:=\cup_n L_n$ and $G_\infty:={\rm Gal}(L_\infty/k)$. The next lemma shows that the following is well defined.
\[
\Theta_{S, \Sigma}^{(\infty)}(u) := \varprojlim_{n}   \Theta_{S, \Sigma}^{(n)}(u) \in \Zp[[G_\infty]][[u]].
\]

\begin{lemma}\label{Theta functoriality}
  Let $L/k$ be a finite  abelian extension with Galois group $G := \Gal(L/k)$. Let $K/k$ be a subextension with $H := \Gal(L/K)$.
We write 
$$\Theta_{S, \Sigma, L/k}(u) \in \Zp[G][u], \quad \Theta_{S, \Sigma, K/k}(u) \in \Zp[G/H][u]$$ 
for the equivariant $L$-functions
attached to the data $(L/k, \Fq, S, \Sigma)$ and  $(K/k, \Fq, S, \Sigma)$, respectively. Then the canonical map $\Zp[G][u] \lra \Zp[G/H][u]$ sends
$\Theta_{S, \Sigma, L/k}(u)$ to $\Theta_{S, \Sigma, K/k}(u)$.
\end{lemma}

\begin{proof}
We write $\pi$ for the canonical map $G \lra G/H$ and also for any map which is naturally induced by $\pi$. 
It is straightforward to verify that for any character $\chi \in \hat G$ one has 
\[
\pi(e_\chi) =
\begin{cases}
  e_\psi, & \text{if } \chi|_H = 1 \text{ and } \chi = \inf_{G/H}^G(\psi), \\
  0, & \text{if } \chi|_H \ne 1.
\end{cases}
\]
Hence, by the inflation invariance of $(S, \Sigma)$-modified Artin $L$-functions, we obtain
\begin{eqnarray*}
  \pi\left( \sum_{\chi \in \hat{G}}L_{S, \Sigma}(\chi, s) e_{\chi^{-1}}\right) &=& 
 \sum_{\psi \in \widehat{G/H}}L_{S, \Sigma}(\psi, s) e_{\psi^{-1}}.
\end{eqnarray*}
It follows that $(\pi(\Theta_{S, \Sigma, L/k}))(q^{-s}) = \Theta_{S, \Sigma, K/k}(q^{-s}) $ for all $s \in \Ce$, and hence we also have 
$\pi(\Theta_{S, \Sigma, L/k}(u)) = \Theta_{S, \Sigma, K/k}(u)$
by \eqref{StickelArtin}.
\end{proof}

We let $X_n \lra Y$ denote the $G_n$-Galois cover of smooth, projective curves defined over $\Fq$ corresponding to $L_n / k$.
We write $\overline{X}_n := X_n \times_{\Fq} \overline\Fq$, $\bar{Y} := Y \times_{\Fq} \overline\Fq$, and also $\bar{S}_n$ for the set of points of $\bar{X}_n$ above
points of $S$. We let $\MSl$ be the Picard $1$-motive associated with the set of data $(\bar{X}_n, \overline\Fq, \bar{S}_n, \emptyset)$ and write
$\Tpl := T_p(\MSl)$ for the $p$-adic Tate module of $\MSl$.  Then, $T_p^{(n)}$ is endowed with a natural structure of $\Bbb  Z_p[G_n][[\Gamma]]$--module. (See \cite[Sec.~3]{GP12}.)

Now, since $\Bbb F_{q^{d_\infty}}$ is the exact field of constants of $L_n$, $\overline X_n$ is going to have $d_\infty$ connected components, all isomorphic to 
$\overline X_n^c:=X_n\times_{\Bbb F_{q^{d\infty}}}\overline\Fq$. We let $T_{p,c}^{(n)}:=T_{p, c}(M_S^{(0)}):=T_p(M_S^{(n), c})$ denote the $p$-adic Tate module of the $1$--motive $M_S^{(n), c}$ associated to the data 
$(\bar{X}_n^c, \overline{\F_q}, \bar{S}_n\cap \overline{X}_n^c ,\, \emptyset)$.\\

For $m \ge n \ge 0$ we write $N_{m/n} \colon L_m^\times \lra L_n^\times$ for the field theoretic norm map and also view 
$N_{m/n} = \sum_{g \in \Gal(L_m/L_n)}g$ as an element of $\Zp[G_m]$. Galois restriction gives isomorphisms
\[
\Gal(\calL_m / \calL_n) \simeq \Gal(L_m / L_n).
\]

By the results of \cite[Sec.~3]{GP12} we have  a natural $\Bbb Z_p[G_{n+1}]$--equivariant, injective morphism
$T_p^{(n)}\hookrightarrow T_p^{(n+1)}.$  We may also consider the norm map
\[
N_{n+1/n}:\, T_p^{(n+1)} \lra \left( T_p^{(n+1)} \right)^{\Gal(L_{n+1}/L_n)}.
\]
By \cite[Th.~3.10]{GP12} this norm map is surjective and, moreover, it is an almost formal consequence of \cite[Th.~3.1]{GP12}
to show that, under the canonical injective morphism  $T_p^{(n)}\hookrightarrow T_p^{(n+1)}$, we can identify the following $\Bbb Z_p[G_n]$--modules
\begin{equation}\label{Tp invariants}
\left( T_p^{(n+1)} \right)^{\Gal(L_{n+1}/L_n)} = \Tpl. 
\end{equation}
Indeed, we recall the diagram of fields above and set for $n \ge 0$
\[
H_n := \Gal(L_n / E).
\]
Then we have natural isomorphisms of $\Bbb Z_p[G_n]$--modules  (see \cite[proof of Th.~3.10]{GP12})
\[
T_p^{(n)} \simeq T_{p,c}^{(n)} \tensor_{\Zp[H_n]} \Zp[G_n].
\]
Therefore, we have the following 
\begin{eqnarray*}
  \left( T_p^{(n+1)} \right)^{\Gal(L_{n+1} / L_n)} &\simeq& 
\left( T_{p,c}^{(n+1)} \tensor_{\Zp[H_{n+1}]} \Zp[G_{n+1}] \right)^{\Gal(L_{n+1} / L_n)} \\
&\stackrel{(*)}=& \left( T_{p,c}^{(n+1)} \right)^{\Gal(L_{n+1} / L_n)}  \tensor_{\Zp[H_{n}]} \Zp[G_{n}] \\
&\stackrel{(**)}=& \left( T_{p,c}^{(n)} \right)  \tensor_{\Zp[H_{n}]} \Zp[G_{n}] \\
&\simeq& T_p^{(n)},
\end{eqnarray*}
where $(*)$ is easy to verify using the fact that $T_{p,c}^{(n+1)}$ is $\Zp[H_{n+1}]$-projective by Theorem \ref{GP Th.4.3}
and $(**)$ is the result of  \cite[Th.~3.1]{GP12}.

\begin{definition}
  We define the Iwasawa--type algebras
  \[
\Lambda_n:=\Bbb Z_p[[G_n\times \Gamma]], \qquad \Lambda := \varprojlim_n \Lambda_n= \Zp[[G_\infty \times \Gamma]],
\]
and  consider the $\Lambda$--module defined by 
\[
T_p(M_S^{(\infty)}) := \varprojlim_n T_p(\MSl),
\]
where the projective limit is taken with respect to norm maps.
\end{definition}
\medskip

The rest of this section is devoted to the proof of Theorem \ref{EMC-I-intro} which we recall for the reader's convenience. 

\begin{theorem}[EMC-I]\label{limit theorem 1} Let $S$ and $\Sigma$ be as above. Then the $\Lambda$--module 
$T_p(M_S^{(\infty)})$ is  finitely generated and torsion and the following hold.
\begin{enumerate}
\item ${\rm pd}_\Lambda \left(T_p(M_S^{(\infty)})\right)=1$.
\item 
$\Fitt_\Lambda\left( T_p(M_S^{(\infty)}) \right) = \Theta_{S, \Sigma}^{(\infty)}(\gamma^{-1}) \cdot \Lambda.$
\end{enumerate}
\end{theorem}
\medskip

The strategy of proof is as follows. First, we obtain a finitely generated, $\Lambda$--projective resolution of length $1$ for $T_p(M_S^{(\infty)})$, as a projective limit of certain $\Lambda_n$--projective resolutions of length $1$ for $T_p(M_S^{(n)})$, for $n\geq 0$, essentially constructed in
\cite{GP12}. This implies that $T_p(M_S^{(\infty)})$ is finitely generated and ${\rm pd}_{\Lambda}\leq 1$. 
Then, we use this construction further to show that 
\begin{equation}\label{lim fitt commute}
\Fitt_\Lambda\left(T_p(M_S^{(\infty)}) \right) =  
\varprojlim_n\, \Fitt_{\Lambda_n} \left( T_p(\MSl) \right).
\end{equation}
Next, we  show that $\Theta_{S, \Sigma}^{(n)}(\gamma^{-1})$ is a non-zero divisor in 
$\Lambda_n$, for all $n\geq 0$ and $\Theta_{S, \Sigma}^{(\infty)}(\gamma^{-1})$ is a non--zero divisor in $\Lambda$. (See Corollary \ref{non-zero divisors} below.)
When combined with Theorem \ref{GP Th.4.3}(2), equality \eqref{lim fitt commute} and Lemma \ref{lim of nzd} below, this leads to the equalities
\[
\Fitt_\Lambda\left( T_p(M_S^{(\infty)}) \right) = \varprojlim_n \left( \Theta_{S, \Sigma}^{(n)}(\gamma^{-1}) \cdot \Lambda_n\right)
= \Theta_{S, \Sigma}^{(\infty)} (\gamma^{-1}) \cdot \Lambda.
\]
Now, the fact that $T_p(M_S^{(\infty)})$ is $\Lambda$--torsion and of projective dimension exactly equal to $1$ follows from the following elementary result.

\begin{lemma}\label{torsion lemma}
Let $R$ be a commutative ring and $X$ a finitely generated $R$--module. Assume that  ${\rm Fitt}_R(X)$ contains a non--zero divisor $f\in R$. Then $X$ is a torsion $R$--module. Consequently, if $X$ is non--zero, then $X$ cannot be a submodule of a free $R$--module and therefore it cannot be $R$--projective.
\end{lemma}
\begin{proof}
From the well--known inclusion ${\rm Fitt}_R(X)\subseteq {\rm Ann}_R(X)$, we conclude that $f\cdot X=0$, which concludes the proof of the Lemma.
\end{proof}

\medskip

From now on, we let $P_n$ denote the Sylow $p$--subgroup of $G_n$ and $\Delta_n$ its complement, so that 
$G_n=P_n\times\Delta_n$, for all $n\geq 0$. Note that since $\Gamma_n=G(L_n/L_0)$ is a $p$--group, $\Delta:=\Delta_n$ does not depend on $n$. Consequently, for all $n\geq 0$, we have 
\begin{equation}\label{Pl Delta decomposition}
  G_n\simeq P_n\times\Delta, \qquad P_{n}/\Gamma_n\simeq P_0.
\end{equation}
Therefore, for all $n\geq 0$, we have an isomorphism of $\Bbb Z_p$--algebras
\[
\Zp[G_n] \simeq \bigoplus_{\chi \in \hat\Delta(\Qp)} \Zp(\chi)[P_n],
\]
given by the usual direct sum of $\chi$--evaluation maps  for $\chi\in\widehat\Delta$. Consequently, any $\Bbb Z_p[G_n]$--module $X$ splits naturally into a direct sum
\[
  X= \bigoplus_{\chi \in \hat\Delta(\Qp)} X^\chi, \qquad\text{where }X^\chi\simeq X\otimes_{\Bbb Z_p[G_n]}\Bbb Z_p(\chi)[P_n].
\]
Since $P_n$ is an abelian $p$--group, the rings $\Bbb Z_p(\chi)[P_n]$ are local rings, for all $n$ and $\chi$ as above.
Further, since projective modules over local rings are free, Theorem \ref{GP Th.4.3} and
\cite[Rem.~2.7]{GP12} imply that we have  isomorphisms of $\Zp(\chi)[P_n]$--module 
\begin{equation}\label{def mchil}
\TpMSl^\chi\simeq  \left( \Zp(\chi)[P_n] \right) ^{m_\chi^{(n)}},
\end{equation}
with integers $m_\chi^{(n)} \ge 0$, for all $\chi$ and $n$ as above.

\begin{lemma}\label{indep of l}
  The non-negative integers  ${m_\chi^{(n)}}$ do not depend on $n$. 
\end{lemma}
\begin{proof}
  Note that since $T_p(M_S^{(n)})$ is $\Zp[G_n]$--projective, taking $\Gal(L_n/L_0)$ fixed points commutes with taking $\chi$-parts.
  Hence (\ref{def mchil}) combined with \eqref{Tp invariants} implies
  $m_\chi^{(n)} = m_\chi^{(0)}$, for all characters  $\chi$ and all $n\geq 0$.
\end{proof}
\medskip

\begin{definition}
We let $m_\chi:=m_\chi^{(n)} = {\rm rank}_{\Bbb Z_p(\chi)[P_n]}T_p(M_S^{(n)})^\chi$, for all $\chi$ and $n$ as above.
\end{definition}
\medskip

In order to simplify notations, for every $\chi\in\widehat\Delta$ and  all $n\geq 0$, we write
\[
R_n:=\Bbb Z_p[G_n], \quad  R_n^\chi := \Zp(\chi)[P_n] , \quad 
T_n = \TpMSl, \quad T_{n}^\chi = \TpMSl^\chi.
\]
We let $P_\infty:=\varprojlim_n P_n$, observe that $G_\infty=P_\infty\times\Delta$, and set
\[
R_\infty:=\Bbb Z_p[[G_\infty]], \quad  R_\infty^\chi := \Zp(\chi)[[P_\infty]] , \quad 
T_\infty := T_p(M_S^{(\infty)}), \quad T_\infty^\chi := T_p(M_S^{(\infty)})^\chi.
\]
Further, we let $\Lambda_n^\chi:=R_n^\chi[[\Gamma]]=\Bbb Z_p(\chi)[[P_n\times \Gamma]]$ and $\Lambda^\chi:=R_\infty^\chi[[\Gamma]]=\Bbb Z_p(\chi)[[P_\infty\times\Gamma]]$. Since 
\[
  \Lambda_n\simeq\bigoplus_{\chi\in\widehat{\Delta}(\Qp)}\Lambda_n^\chi, \quad
  \Lambda\simeq \bigoplus_{\chi\in\widehat{\Delta}(\Qp)}\Lambda^\chi,
\]
via the usual character--evaluation maps, the $\Lambda_n$--modules $\Lambda_n^\chi$ and the $\Lambda$--modules $\Lambda^\chi$ are projective and cyclic, for all characters $\chi$ as above.\\

Now, we fix a character $\chi$ as above and, for a given $n\geq 0$, we fix an  $R_n^\chi$-basis of $T_n^\chi$:
\[
x_1^{(n)}, \ldots, x_{m_\chi}^{(n)}.
\]
We let $A_{\gamma}^{(n),\chi} \in \Gl_{m_\chi}(R_n^\chi)$ be the matrix associated to the action of $\gamma$ on $T_n^\chi$
  with respect to the fixed basis. Let $\Phi_\gamma^{(n), \chi}$ be the $R_n^\chi[[\Gamma]]$-linear endomorphism of $R_n^\chi[[\Gamma]]^{m_\chi}$
of matrix
\[
1 - \gamma^{-1} A_\gamma^{(n), \chi} \in M_{m_\chi}(R_n^\chi[[\Gamma]])
\]
with respect to the canonical $R_n^\chi[[\Gamma]]$-basis $e_1^{(n)}, \ldots, e_{m_\chi}^{(n)}$ of $R_n^\chi[[\Gamma]]^{m_\chi}$.
By the proof of \cite[Prop.~4.1]{GP12}, in particular (6) of loc.cit.,  combined with Corollary \ref{non-zero divisors}(2) below, we have an exact sequence of $R_n^\chi[[\Gamma]]$--modules
\begin{equation}\label{Lambda-ell-chi-ses}
0\lra R_n^\chi[[\Gamma]]^{m_\chi} \stackrel{\Phi_\gamma^{(n), \chi}}\lra R_n^\chi[[\Gamma]]^{m_\chi}\stackrel{\pi_n^\chi}\lra T_n^\chi \lra 0,
\end{equation}
where $\pi_n^\chi$ is defined by $\pi_n^\chi(e_i^{(n)}) = x_i^{(n)}$. \\

Next, we show that, given an $R_n^\chi$--basis $x_1^{(n)}, \ldots, x_{m_\chi}^{(n)}$ for $T_n^\chi$, we can choose an $R_{n+1}^\chi$-basis
$
x_1^{(n+1)}, \ldots, x_{m_\chi}^{(n+1)}
$
of $T_{n+1}^\chi$, such that we have a commutative diagram
 \begin{equation}\label{comm GP diagram}
\xymatrix{
0\ar[r]&R_n^\chi[[\Gamma]]^{m_\chi} \ar[r]^{\Phi_\gamma^{(n), \chi}} & R_n^\chi[[\Gamma]]^{m_\chi} \ar[r]^{\pi_n^\chi} & T_n^\chi \ar[r] & 0 \\
0\ar[r] &R_{n+1}^\chi[[\Gamma]]^{m_\chi} \ar[r]^{\Phi_\gamma^{(n+1), \chi}} \ar@{>>}[u]^{\varphi_{n+1/n}} & R_{n+1}^\chi[[\Gamma]]^{m_\chi} \ar[r]^{\pi_{n+1}^\chi}  \ar@{>>}[u]^{\varphi_{n+1/n}} & 
T_{n+1}^\chi \ar[r]  \ar@{>>}[u]^{N_{n+1/n}} & 0. 
}
\end{equation}
Here the vertical maps on the left and in middle are defined by $e_i^{(n+1)} \mapsto e_i^{(n)}$ and the canonical (componentwise) projections $R_{n+1}^\chi\to R_n^\chi$.

To that end, we start with an arbitrary 
$R_{n+1}^\chi$-basis
\[
y_1^{(n+1)}, \ldots, y_{m_\chi}^{(n+1)}
\]
of $T_{n+1}^\chi$ and show how to modify it so that diagram (\ref{comm GP diagram}) commutes.
Since the module  $T_{n+1}$ is $G_{n+1}$--cohomologically trivial (see \cite[Th.~3.10]{GP12}), by (\ref{Tp invariants}) above
we have
\[
N_{n+1/n}\left(T_{n+1}^\chi \right) =  \left(T_{n+1}^\chi \right)^{\Gal(L_{n+1}/L_n)} = T_n^\chi
\]
and therefore
\[
\left\{ N_{n+1/n}\left( y_1^{(n+1)} \right), \ldots, N_{n+1/n}\left( y_{m_\chi}^{(n+1)} \right) \right\}
\]
is an $R_n^\chi$-basis of $T_ n^\chi$. Let $U_n \in \Gl_{m_\chi}(R_n^\chi)$ denote the matrix such that
\[
\left(
  \begin{array}{c}
    x_1^{(n)} \\ \vdots \\ x_{m_\chi}^{(n)}
  \end{array}
\right) =
U_n 
\left(
  \begin{array}{c}
    N_{n+1/n}\left( y_1^{(n+1)} \right) \\ \vdots \\ N_{n+1/n}\left( y_{m_\chi}^{(n+1)} \right) 
  \end{array}
\right).
\]
Let $U_{n+1} \in M_{m_\chi}(R_{n+1}^\chi)$ be such
that $ \varphi_{n+1/n}(U_{n+1}) = U_n$. Actually, by the next lemma, $U_{n+1}$ is an invertible matrix.

\begin{lemma}
  Let $\varphi \colon S \lra R$ be a morphism of commutative local rings, i.e. $\varphi(\frm_S) \sseq \frm_R$, where $\frm_S$ and $\frm_R$ are the corresponding maximal ideals.
 Let $U \in M_m(R)$,  $V \in M_m(S)$ be matrices such that $\varphi(V) = U$. Then:
\[
U \in \Gl_m(R) \iff V \in \Gl_m(S).
\]
\end{lemma}
\begin{proof}
  If $VW = 1$ with $W \in  M_m(S)$, then $1 = \varphi(VW) = \varphi(V) \cdot \varphi(W) = U \cdot \varphi(W)$, i.e., $U^{-1} = \varphi(W)$.
Conversely suppose that $V \not\in \Gl_m(S)$. Then $\det_S(V) \in \frm_S$,
 and hence
\[
\det\nolimits_R(U) = \det\nolimits_R(\varphi(V)) = \varphi( \det\nolimits_R(V) ) \in \frm_R,
\]
contradicting $U \in \Gl_m(R)$.
\end{proof}

Now, since $U_{n+1}$ is invertible, $\{x_1^{(n+1)}, \dots, x_{m_\chi}^{(n+1)}\}$ defined by
\[
\left(
  \begin{array}{c}
    x_1^{(n+1)} \\ \vdots \\ x_{m_\chi}^{(n+1)}
  \end{array}
\right) =
U_{n+1} 
\left(
  \begin{array}{c}
     y_1^{(n+1)}  \\ \vdots \\  y_{m_\chi}^{(n+1)}  
  \end{array}
\right).
\]
is an $R_{n+1}^\chi$--basis of $T_{n+1}^\chi$. Then the right hand square of (\ref{comm GP diagram}) commutes because
\begin{equation}\label{norm comp basis}
  \left(
  \begin{array}{c}
    N_{n+1/n}x_1^{(n+1)} \\ \vdots \\ N_{n+1/n}x_{m_\chi}^{(n+1)}
  \end{array}
\right) =
U_{n+1} 
\left(
  \begin{array}{c}
     N_{n+1/n}y_1^{(n+1)}  \\ \vdots \\  N_{n+1/n}y_{m_\chi}^{(n+1)}  
  \end{array}
\right)
=
U_{n} 
\left(
  \begin{array}{c}
     N_{n+1/n}y_1^{(n+1)}  \\ \vdots \\  N_{n+1/n}y_{m_\chi}^{(n+1)}  
  \end{array}
\right) = \left(
  \begin{array}{c}
    x_1^{(n)} \\ \vdots \\ x_{m_\chi}^{(n)}
  \end{array}
\right). 
\end{equation}
Let $\mu_\gamma$ denote multiplication by $\gamma$ in $T_n^\chi$. By the definition of $A_{\gamma}^{(n), \chi}$, one has 
\begin{equation}\label{matrix coeff}
\mu_\gamma\left( x_i^{(n)} \right) = \sum_{j=1}^{m_\chi}A_{\gamma,ij}^{(n), \chi} x_j^{(n)}, \quad\text{for all } n\ge 0\text{ and } 1\leq i\leq m_\chi.
\end{equation}
To prove commutativity of the left hand square of (\ref{comm GP diagram}) one has to show
\begin{equation}\label{coeffientwise}
  \varphi_{n+1/n}\left( A_{\gamma,ij}^{(n+1), \chi} \right) =  A_{\gamma,ij}^{(n), \chi}.
\end{equation}
Since $\mu_\gamma$ is an $R_n^\chi$--linear map, 
it follows from (\ref{norm comp basis}) and (\ref{matrix coeff}) that 
\begin{eqnarray*}
  \mu_\gamma\left( x_i^{(n)} \right) &=& \mu_\gamma\left( N_{n+1/n}x_i^{(n+1)} \right)  \\
&=& N_{n+1/n} \left(  \sum_{j=1}^{m_\chi}A_{\gamma,ij}^{(n+1), \chi} x_j^{(n+1)} \right)   \\
&=&  \sum_{j=1}^{m_\chi}  \varphi_{n+1/n} \left(  A_{\gamma,ij}^{(n+1), \chi}\right)N_{n+1/n} \left(  x_j^{(n+1)} \right)   \\
&=&  \sum_{j=1}^{m_\chi} \varphi_{n+1/n} \left(  A_{\gamma,ij}^{(n+1), \chi}\right)x_j^{(n)}, 
\end{eqnarray*}
and this, in turn, immediately implies (\ref{coeffientwise}).\\

Now, we start with an $R_0^\chi$--basis $x_1^{(0)}, \dots, x_{m_\chi}^{(0)}$ for $T_0^\chi$ and use the procedure above inductively to construct 
$R_n^\chi$--bases $x_1^{(n)}, \dots, x_{m_\chi}^{(n)}$ for $T_n^\chi$ so that \eqref{comm GP diagram} commutes, for all $n\geq 0$.  Therefore, we can take a projective limit as $n\to\infty$ in \eqref{comm GP diagram}.  The Mittag-Leffler property
(see \cite[Prop.~9.1]{Hartshorne}) implies that we obtain an exact sequence of $\Lambda^\chi$--modules
 \begin{equation}\label{Lambda-chi-ses}
\xymatrix{
0 \ar[r] & (\Lambda^\chi)^{m_\chi} \ar[rr]^{\Phi_\gamma^{(\infty), \chi}} &&
(\Lambda^\chi)^{m_\chi} \ar[rr]^{\qquad \pi_\infty^\chi} && T_\infty^\chi\ar[r] & 0,
}
\end{equation}
where $\Phi_\gamma^{(\infty), \chi}:=\varprojlim_n \Phi_\gamma^{(n), \chi}$ and $\pi_\infty^\chi:=\varprojlim_n\pi_n^\chi$.
By \eqref{coeffientwise}, we may define the following matrix
\[
A_\gamma^{(\infty), \chi} := \{ A_\gamma^{(n), \chi} \}_{n \ge 0} \in \varprojlim_n \Gl_{m_\chi}(R_n^\chi)=\Gl_{m_\chi}(R_\infty^\chi).
\]
Consequently, the map $\Phi_\gamma^{(\infty), \chi}$ has matrix $(1-\gamma^{-1}A_\gamma^{(\infty), \chi})$ in the standard basis of $(\Lambda^\chi)^{m_\chi}$, for all characters $\chi$ as above.\\

If we take the direct sum of \eqref{Lambda-chi-ses} over all $\chi$ we obtain the exact sequence of $\Lambda$--modules
\begin{equation}\label{Lambda-ses}
\xymatrix{
0 \ar[r] & \bigoplus_{\chi}(\Lambda^\chi)^{m_\chi} \ar[rr]^{\Phi_\gamma^{(\infty)}} && \bigoplus_\chi(\Lambda^\chi)^{m_\chi} \ar[rr]^{\pi_\infty} && T_\infty \ar[r] &  0,
}
\end{equation}
where $\Phi_\gamma^{(\infty)}:=(\Phi_\gamma^{(\infty), \chi})_\chi$ and $\pi_\infty:=(\pi_\infty^\chi)_\chi$. 
The exact sequence above shows that the $\Lambda$--module $T_\infty = T_p(M_S^{(\infty)})$ is finitely generated, of projective dimension at most $1$. \\

Moreover, for all $\chi$ we have the following equalities
\begin{eqnarray}
\Fitt_{\Lambda^\chi}\left( T_\infty^\chi \right) &\stackrel{(*)}=& \det\nolimits_{\Lambda^\chi} \left( 1 - \gamma^{-1}A_\gamma^{(\infty), \chi} \right) \cdot \Lambda^\chi \nonumber \\ 
                                                 &\stackrel{(**)}=& \varprojlim_n  \left( \det\nolimits_{\Lambda_n^\chi} \left( 1 - \gamma^{-1}A_\gamma^{(n), \chi} \right) \cdot \Lambda_n^\chi\right)
                                                                    \label{Fitt eqs I} \\
&\stackrel{(***)}=& \varprojlim_n\,  \Fitt_{\Lambda_n^\chi}\left(T_n^\chi\right). \nonumber
\end{eqnarray}
Above, equality (*) follows from \eqref{Lambda-chi-ses}, equality (**) follows from Lemma \ref{lim of nzd} and  Corollary \ref{non-zero divisors}(1)  
below, and (***) follows from \eqref{Lambda-ell-chi-ses}. 

Now, we take the direct sum over all $\chi$ of equality \eqref{Fitt eqs I} to obtain
\begin{eqnarray*}\label{eqn:one}
  \Fitt_\Lambda\left( T_\infty \right)
&=& \varprojlim_n\,  \Fitt_{\Lambda_n} \left(  T_n\right )   \\
&\stackrel{(*)}=& \varprojlim_n \left(\Theta_{S, \Sigma}^{(n)}(\gamma^{-1}) \cdot \Lambda_n\right) \\
 &\stackrel{(**)}=& \Theta_{S, \Sigma}^{(\infty)}(\gamma^{-1}) \cdot \Lambda, 
\end{eqnarray*}
where $(*)$ is one of the main results of Greither and Popescu (see Theorem \ref{GP Th.4.3}(2) above) and  $(**)$ is Corollary \ref{non-zero divisors}(4)  below. 

Now, the fact that $T_\infty$ is torsion and has projective dimension exactly equal to $1$ over $\Lambda$ follows from equality
\eqref{Fitt eqs I} above, Corollary  \ref{non-zero divisors}(5), and Lemma \ref{torsion lemma}.\\

This concludes the proof of Theorem \ref{limit theorem 1}, save for the technical results which imply the injectivity of the maps
$\Phi_\gamma^{(n), \chi}$ and therefore the exactness of \eqref{Lambda-ell-chi-ses}, as well as both equalities (**) above. We state and prove these technical results below.\\

\begin{lemma}\label{lim of nzd}
  Let $(R_m, \pi_{m,n})$ be a projective system of commutative rings and set $R_\infty := \varprojlim_m R_m$. Let 
$\alpha_\infty := \{ \alpha_m \}_m \in R_\infty $ be a coherent sequence of non-zero divisors. Then
\[
\varprojlim_m\, (\alpha_m R_m) = \alpha_\infty R_\infty.
\]
\end{lemma}

\begin{proof}
  Let $\{ \alpha_m r_m \}_m \in \varprojlim_m \alpha_m R_m$. Then, for $m > n$,
\[
\alpha_n r_n = \pi_{m,n}(\alpha_{m} r_{m}) = \alpha_n \pi_{m,n}(r_{m}).
\]
Since $\alpha_n$ is a non-zero divisor by assumption, this implies $r_n = \pi_{m,n}(r_{m})$, and hence
$\{ \alpha_m r_m \}_m = \alpha_\infty r_\infty \in  \alpha_\infty R_\infty$ with $r_\infty := \{ r_m \}_m $. The opposite containment is obvious
(and true without the assumption that the $a_m$ are non-zero divisors).
\end{proof}

\begin{lemma}\label{nzd polynomials}
Let $R$ be $R_n^\chi$ or $R_n$. Let $f \in R[\gamma] \sseq R[[\Gamma]]$ be a polynomial in $\gamma$ such that the leading coefficient is a unit in $R$. 
Then $f$ is a non-zero divisor in $R[[\Gamma]]$.
\end{lemma}

\begin{proof}
 We write
\[
f = \lambda_d \gamma^d + \lambda_{d-1} \gamma^{d-1} + \ldots + \lambda_{1} \gamma + \lambda_0
\]
with $\lambda_j \in R$ and $\lambda_d \in R^\times$. We argue by
contradiction and suppose that $f$ is a zero divisor in  $R[[\Gamma]]$. Then there exists $g \in  R[[\Gamma]]$ 
such that $fg = 0$ and $g \ne 0$. Let $f = \{f_m\}_{m \in \NN}$ and  $g = \{g_m\}_{m \in \NN}$
with $f_m, g_m \in R[\Gamma / \Gamma^m]$. Then there exists $N \in \NN$ such that for all $m \in \NN$ 
\[
  f_{Nm} g_{Nm} = 0, \quad f_{Nm} \ne 0 \ne g_{Nm}
\]
in $R[\Gamma / \Gamma^{Nm}]$. In particular, for all $k \ge 0$, we have
\begin{equation}\label{zero divisor eq 1}
  f_{Np^k} g_{Np^k} = 0, \quad f_{Np^k} \ne 0 \ne g_{Np^k}.
\end{equation}
Write $N = Mp^a$ with $p \nmid M$. Since $\Gamma\simeq\widehat {\Bbb Z}$ (the profinite completion of $\Bbb Z$), we have the following group isomorphisms
\[
\Gamma/\Gamma^{Np^k} = \Gamma/\Gamma^{Mp^{a+k}} \simeq \Gamma/\Gamma^M \times \Gamma/\Gamma^{p^{a+k}} 
\simeq \Ze/M\Ze \times \Ze/p^{a+k}\Ze.
\]
Further, we have a surjective topological group morphism 
$$\Gamma\twoheadrightarrow \Gamma/\Gamma^M \times \varprojlim_k\Gamma/\Gamma^{p^{a+k}}\simeq \Bbb Z/M\Bbb Z\times\Bbb Z_p, \quad \gamma\to (\gamma_M, \gamma_p).$$ 
It follows that $\varprojlim_k R[\Gamma/\Gamma^{Np^k} ] \simeq R[\Gamma/\Gamma^M][[t]]$, where $\gamma_p$ maps to $1+t$.
Thus, we obtain a $\Bbb Z_p$--algebra homomorphism $\varphi$ defined by the following composition of maps
\[
\varphi:\,R[[\Gamma]]\overset{\pi_{M,p}}\longrightarrow\varprojlim_k R[\Gamma/\Gamma^{Np^k} ] \simeq R[\Gamma/\Gamma^M][[t]] \hookrightarrow
\bigoplus_\psi \Zp[\psi][[t]]
\]
where $\psi$ runs through the $\overline{\Bbb Q_p}$-valued characters of $(G_n \times \Gamma/\Gamma^M)$ modulo
the action of $\Gal(\overline{\Bbb Q_p}/\Qp)$ (and such
that $\psi|_\Delta = \chi$ if $R = R^\chi_{n}$). Clearly, $\varphi$ maps $f$ to $(\psi(\lambda_d\gamma_M^d) (1+t)^d + \ldots)_\psi$.  
As $\lambda_d \in R^\times$ by assumption, we have $\psi(\lambda_d\gamma_M^d)\in\Bbb Z_p[\psi]^\times$, so  the $\psi$-component of $\varphi(f)$ is non-zero, for all $\psi$. Since $\Bbb Z_p[\psi][[t]]$ is an integral domain, for all $\psi$, this implies that 
$\pi_{M,p}(f)$ is not a zero--divisor in $\varprojlim_k R[\Gamma/\Gamma^{Np^k} ]$. However, this contradicts equalities \eqref{zero divisor eq 1}, which concludes the proof of the Lemma.

\end{proof}

\begin{coro}\label{non-zero divisors} The following hold for all  $n \ge 0$ and all $\chi\in\widehat{\Delta}$.

\begin{enumerate}

\item The  element $\det\nolimits_{R_n^\chi[[\Gamma]]}\left( 1 - \gamma^{-1}A_\gamma^{(n), \chi} \right)$
is a  non-zero divisor in $R_n^\chi[[\Gamma]]$.

\item The map $\Phi_\gamma^{(n), \chi} \colon R_n^\chi[[\Gamma]]^{m_\chi} \lra  R_n^\chi[[\Gamma]]^{m_\chi}$ is injective.

\item The  element $  \Theta_{S, \Sigma}^{(n)}(\gamma^{-1})$ is a non-zero divisor in 
$\Lambda_n = R_n[[\Gamma]] $.

\item We have an equality $\Theta_{S, \Sigma}^{(\infty)}(\gamma^{-1})\cdot\Lambda=\varprojlim_n\left( \Theta_{S, \Sigma}^{(n)}(\gamma^{-1})\cdot\Lambda_n\right).$
\item The element $\Theta_{S, \Sigma}^{(\infty)}(\gamma^{-1})$ is a non--zero divisor in $\Lambda$.
\end{enumerate}
\end{coro}
\begin{proof}
 (1) Observe that $\gamma^{m_\chi} \det_{R_n^\chi[[\Gamma]]}\left(1 - \gamma^{-1} A_\gamma^{(n), \chi} \right)$ is a polynomial
in $R_n^\chi[\gamma]$ of degree $m_\chi$ and leading coefficient $1$. Hence part (1) follows immediately from Lemma \ref{nzd polynomials} above.

(2) is a consequence of (1) and the following fact: let $R$ be a commutative ring, $A \in M_n(R)$ and suppose that $\det_R(A)$ is not a zero divisor.
Then $A \colon R^n \lra R^n$ (defined with respect to the standard basis) is injective. Indeed, let $A^*$ be the adjoint matrix. 
Then $$AA^* = A^*A = \rm{det}_R(A)\cdot {\rm id},$$ 
and therefore
$A^\ast\circ A$ is injective and, as a consequence, $A$ is injective

(3)  We apply results of \cite{GP12}, in particular Propositions 4.8 and 4.10. We can express the image $\Theta_{S, \Sigma}^{(n), \chi}(u)$
of  $\Theta_{S, \Sigma}^{(n)}(u)$ in $R_n^\chi[u]$ as a 
product of two polynomials $P^\chi(u), Q^\chi(u) \in R_n^\chi[u]$, such that $Q^\chi(\gamma^{-1}) \in R_n^\chi[[\Gamma]]^\times$. It thus suffices to
show that $P^\chi(\gamma^{-1})$ is a non-zero divisor. By \cite[Prop.~4.8 a)]{GP12} we have
\[
P^\chi(u) = \det\nolimits_{R_n^\chi}\left(1-\gamma u \mid T_p(M_{\bar{S}, \bar{\Sigma}}^{(n)})^\chi\right),
\]
so the element $\gamma^{m_\chi} P^\chi(\gamma^{-1})$ is a polynomial with 
leading coefficient $1$. Lemma \ref{nzd polynomials} implies that $\Theta_{S, \Sigma}^{(n), \chi}(\gamma^{-1})$ is  a non-zero divisor in $\Lambda_n^\chi$, for all $\chi$. Therefore, $\Theta_{S, \Sigma}^{(n)}(\gamma^{-1})=(\Theta_{S, \Sigma}^{(n), \chi}(\gamma^{-1}))_\chi$ is a non--zero divisor in $\Lambda_n=\oplus_\chi\Lambda_n^\chi$.

(4) Apply part (3) combined with Lemmas \ref{Theta functoriality} and \ref{lim of nzd}.

(5) This follows immediately from (3), as $\Theta_{S, \Sigma}^{(\infty)}(\gamma^{-1})=\varprojlim_n\Theta_{S, \Sigma}^{(n)}(\gamma^{-1})$
in $\Lambda=\varprojlim_{n}\Lambda_n$.
\end{proof}

\medskip

\subsection{Co-descent to $L_\infty/k$}
In what follows, for every $n\geq 0$, we denote by $D^0(L_n)$ and $D_S^0(L_n)$ the $\Bbb Z [G_n]$--modules of
divisors of degree $0$ in $L_n$ and divisors of degree $0$ supported 
at primes above $S$ in $L_n$, respectively.
By $D_S(L_n)$ we denote the $S$--supported divisors of $L_n$ of arbitrary degree. Note that the degree is computed relative to $\Bbb F_q$.
Also, $U_S^{(n)}$ denotes the $\Bbb Z[G_n]$--module of $S$--units in $L_n$
(i.e. elements $f\in L_n^\times$ whose divisor ${\rm div}(f)$ is in $D_S^0(L_n)$). Finally,
$X_S^{(n)}$ denotes the $\Bbb Z[G_n]$--module of divisors of $L_{n}$ supported at primes above $S$ and
of {\it formal degree} $0$, i.e., formal sums $\sum_{v\in S(L_n)}n_v\cdot v$, with $n_v\in \Bbb Z$ and $\sum_v n_v=0$.\\

  By slightly  generalizing the results in \cite{GPff}, our Proposition \ref{smallS-proposition} and Remark \ref{RW-Tate-remark} in the Appendix applied for $L=L_n$,
  for every $n\geq 0$, provide us with canonical exact  sequences of $\Bbb Z_p[G_n]$--modules
\begin{equation}\label{Tate seq}
 0 \lra U_S^{(n)} \tensor_\Ze \Zp \lra T_p(M_S^{(n)}) \stackrel{1- \gamma}\lra  T_p(M_S^{(n)}) \lra \Nabla_S^{(n)} \lra 0.
\end{equation}
Here $\Nabla_S^{(n)}:=T_p(M_S^{(n)})_\Gamma$ sits in a short exact sequence of $\Bbb Z_p[G_n]$--modules
\begin{equation}\label{RW-ses-ell}
0 \lra \Pic^0_S(L_n) \tensor_\Ze \Zp \lra \Nabla_S^{(n)} \lra\widetilde X_S^{(n)} \lra 0,
\end{equation}
and $\widetilde X_S^{(n)}$ (defined precisely in the appendix) sits itself in a short exact sequence 
\begin{equation}\label{tilde-x-ses-ell}0\to \Bbb Z_p/d_S^{(n)}\Zp \to \widetilde X_S^{(n)}\to X_S^{(n)}\otimes\Bbb Z_p\to 0,
\end{equation}
where $d_S^{(n)}\Bbb Z:={\rm deg}(D_S(L_n))$ and $G_n$ acts trivially on $\Bbb Z_p/d_S^{(n)}\Zp$. In particular, note that if $S(L_n)$ (the set of places of $L_n$ sitting above places in $S$) contains a prime of degree coprime to $p$, then we have 
$\widetilde X_S^{(n)}= X_S^{(n)}\otimes\Bbb Z_p$. 
\\

Exact sequence \eqref{Tate seq} combined with Theorem \ref{GP Th.4.3} above, gives the following. 
\begin{coro}\label{Fitt-Nabla-ell}
  For all finite non-empty sets $\Sigma$ with $S \cap \Sigma = \emptyset$ and all $n \in \NN$ we have
\[
{\rm Fitt}_{\Bbb Z_p[G_n]}\left(\Nabla_S^{(n)}\right) = \Theta_{S, \Sigma}^{(n)}(1) \cdot \Zp[G_n].
\]
\end{coro}
\begin{proof}
Exact sequence \eqref{Tate seq} gives an isomorphism of $\Bbb Z_p[G_n]$--modules
$$\Nabla_S^{(n)}\simeq T_p(M_S^{(n)})\otimes_{\Bbb Z_p[G_n][[\Gamma]]}\Bbb Z_p[G_n],$$
where $\Zp[G_n]$ is viewed as a $\Bbb Z_p[G_n][[\Gamma]]$--algebra via the unique $\Bbb Z_p[G_n]$--algebra morphism
$\pi:\Bbb Z_p[G_n][[\Gamma]]\twoheadrightarrow\Bbb Z_p[G_n]$ which takes $\gamma\to 1$.
Since Fitting ideals commute with extension of scalars, this gives an equality of $\Bbb Z_p[G]$--ideals 
$${\rm Fitt}_{\Bbb Z_p[G_n]}\left(\Nabla_S^{(n)}\right)=\pi\left({\rm Fitt}_{\Bbb Z_p[G_n][[\Gamma]]}(T_p(M_S^{(n)})\right).$$
Now, the Corollary follows from Theorem \ref{GP Th.4.3}.
\end{proof}
\medskip

In what follows, we set
\[
 U_S^{(\infty)} := \varprojlim_n (U_S^{(n)} \tensor_\Ze \Zp), \quad \Nabla_S^{(\infty)} :=  \varprojlim_n \Nabla_S^{(n)},
\]
where both limits are taken with respect to norm maps. 
\begin{lemma}\label{infty Tate seq lemma}
The sequence (\ref{Tate seq}) stays exact when we pass to the limit, i.e.
\begin{equation}\label{infty Tate seq}
0 \lra U_S^{(\infty)} \lra T_p(M_S^{(\infty)}) \stackrel{1- \gamma}\lra  T_p(M_S^{(\infty)}) \lra \Nabla_S^{(\infty)} \lra 0
\end{equation}
is an exact sequence of $\Lambda$-modules.
\end{lemma}
\begin{proof}
  We set $W^{(n)} := (1-\gamma) T_p(M_S^{(n)}) $. Then the functor $\varprojlim_n$ is exact on 
\begin{equation*}
\xymatrix{
0 \ar[r] &  U_S^{(n+1)} \tensor_\Ze \Zp \ar[r] \ar[d]^{N_{n+1/n}} &  T_p(M_S^{(n+1)}) \ar[r] \ar[d]^{N_{n+1/n}}
& W^{(n+1)}   \ar[r] \ar[d]^{N_{n+1/n}} & 0 \\
0 \ar[r] &  U_S^{(n)} \tensor_\Ze \Zp \ar[r]  &  T_p(M_S^{(n)}) \ar[r] & W^{(n)}   \ar[r]  & 0 \\
}
\end{equation*}
because $ U_S^{(n)} \tensor_\Ze \Zp $ is a finitely generated $\Zp$-module and therefore the projective system $\{U_S^{(n)} \tensor_\Ze \Zp\}_n$
satisfies the Mittag-Leffler condition. 

Since $N_{n+1/n} \colon  T_p(M_S^{(n+1)}) \lra T_p(M_S^{(n)})$ is surjective, the map $N_{n+1/n} \colon  W^{(n+1)} \lra W^{(n)}$
is also surjective. Hence, as above,  $\varprojlim_n$ is exact on 
\begin{equation*}
\xymatrix{
0 \ar[r] &  W{(n+1)} \ar[r] \ar[d]^{N_{n+1/n}} &  T_p(M_S^{(n+1)}) \ar[r] \ar[d]^{N_{n+1/n}}
& \Nabla_S^{(n+1)}   \ar[r] \ar[d]^{N_{n+1/n}} & 0 \\
0 \ar[r] &  W^{(n)} \ar[r]  &  T_p(M_S^{(n)}) \ar[r] & \Nabla_S^{(n)}   \ar[r]  & 0 \\
}
\end{equation*}
We can now glue the two short exact sequences at the $\infty$-level.
\end{proof}
\medskip

The following is an equivariant Iwasawa main conjecture--type result along the Drinfeld module (geometric) tower $L_\infty/k$, for the $\Bbb Z_p[[G_\infty]]$--module $\Nabla_S^{(\infty)}$, see Theorem \ref{EMC-II-intro} of the introduction.

\begin{theorem}[EMC-II]\label{limit theorem 2} For any finite, non-empty set $\Sigma$ of primes in k, disjoint from $S$,
  the following hold.
\begin{enumerate}
\item $\Nabla_S^{(\infty)}$ is a finitely generated, torsion $\Bbb Z_p[[G_\infty]]$--module of projective dimension $1$.
\item $\Fitt_{\Zp[[G_\infty]]}\left(\Nabla_S^{(\infty)}\right) = \Theta_{S, \Sigma}^{(\infty)}(1) \cdot \Zp[[G_\infty]].$
\end{enumerate}
\end{theorem}
\begin{proof} Part (2) is an immediate consequence of exact sequence \eqref{infty Tate seq} and Theorem \ref{limit theorem 1} above. (Repeat the arguments in the proof of Corollary \ref{Fitt-Nabla-ell}.)

Part (1) is Proposition \ref{nabla torsion prop} in \S3.5 below, which itself is a consequence of the fact that $\Theta_{S, \Sigma}^{(\infty)}(1)$ is a non--zero divisor in $\Bbb Z_p[[G_\infty]]$, as proved in Proposition \ref{not a zero divisor prop} below. 
\end{proof}
\medskip

\subsection{Results on ideal class groups}
  We conclude this section by deriving an Iwasawa main conjecture in the spirit of \cite[Th.~1.4]{Angles}
  for the classical $\Bbb Z_p[[G_\infty]]$--module 
{$$\frak X_{p}^{(\infty)}:=\varprojlim_n\left( \Pic^0(L_n) \tensor_\Ze \Zp\right),$$}
where the projective limit is taken with respect to the usual norm maps.

  As in \eqref{Pl Delta decomposition} we let $\Delta$ denote the maximal subgroup of $G_0$ whose order is not divisible by $p$.
  Since $G(L_\infty/L_0)$ is a pro--$p$--group, we have $G_\infty \simeq \Delta\times G_\infty^{(p)}$, 
  where $G^{(p)}_\infty$ is the maximal pro--$p$ subgroup of $G_\infty$. As a consequence, we can view the element
  $e_\Delta:=\frac{1}{|\Delta|}\sum_{\delta\in\Delta}\delta$  as an idempotent in $\Bbb Z_p[[G_\infty]]$. This allows us to define the functor
$$M\mapsto M^\sharp:=(1-e_{\Delta})\cdot M$$
from the category of $\Bbb Z_p[[G_\infty]]$--modules to the category of modules over the quotient ring
$$\Bbb Z_p[[G_\infty]]^\sharp = (1-e_{\Delta})\Bbb Z_p[[G_\infty]].$$ 
Note that since $M=e_{\Delta}M\oplus M^\sharp$, for every $M$ as above,  the functor $M\mapsto M^\sharp$ is exact.\\

The study of $\frak X_{p}^{(\infty)}$ requires some additional hypotheses. We list them below. 
\begin{itemize}
\item [(a)] $\frf = \fre$.
\item [(b)] $\frp$ does not split in $H_\fre / k$.
\item[(c)] $p\nmid [H_\fre:k]=h_k d_\infty$.
\item[(d)] $p\nmid{\rm deg}(\frp)$.
\end{itemize}
Note that (a), (b) and (c) are satisfied in the basic example of the Carlitz module, see \S\ref{Carlitz subsection}.

\begin{theorem}[EMC-III]\label{limit theorem 3} Under hypotheses (a)--(d), the following hold for $S=\{\frak p\}$ and all nonempty sets $\Sigma$, disjoint from $S$.
\begin{enumerate}
\item $\frak X_{p}^{(\infty)}$ is a torsion $\Bbb Z_p[[G_\infty]]$--module of projective dimension $1$.
\item $\Fitt_{\Zp[[G_\infty]]}\left(\frak X_{p}^{(\infty)}\right) = \Theta_{S, \Sigma}^{(\infty)}(1) \cdot \Zp[[G_\infty]].$
\end{enumerate}
\end{theorem}
\begin{proof}
Since $L_n = H_{\frp^{n+1}}$ is unramified outside $\frp$, the set $S = \{ \frp \}$ satisfies all the desired requirements. 
In general,  each divisor of $\frp$ in $H_\fre$ is totally ramified in $H_{\frp^{n+1}}/ H_\fre$. Hence, hypotheses (a)--(b) imply that
$D_S^0(L_n) = 0$ and $X_S^{(n)} = 0$. Further, hypotheses (c)--(d) imply that if $\frak p_n$ is the unique prime in
$L_n$ sitting above $\frak p$, then 
$d_S^{(n)}={\rm deg}(\frak p_n)=[H_{\frak e}:k]\cdot{\rm deg}(\frak p)$
is not divisible by $p$. Consequently, \eqref{RW-ses-ell} and \eqref{tilde-x-ses-ell} give us
\[
  \Nabla_S^{(n)} = \Pic^0(L_n)\otimes\Bbb Z_p, \qquad \Nabla_S^{(\infty)}=\frak X_{p}^{(\infty)},
\]
and the result follows from Theorem \ref{limit theorem 2} above.
\end{proof}

Under the milder hypotheses (a)--(b), both satisfied in the basic case of the Carlitz module, a similar result holds,
away from the trivial character of $\Delta$.
\begin{theorem}[${\rm EMC-III}^\sharp$]\label{limit theorem 3 sharp} Under hypotheses (a)--(b),
  the following hold for $S=\{\frak p\}$ and all nonempty sets $\Sigma$, disjoint from $S$.
\begin{enumerate}
\item $\frak X_{p}^{(\infty), \sharp}$ is a torsion $\Bbb Z_p[[G_\infty]]^\sharp$--module of projective dimension $1$.
\item $\Fitt_{\Zp[[G_\infty]]^\sharp}\left(\frak X_{p}^{(\infty), \sharp}\right) = \Theta_{S, \Sigma}^{(\infty)}(1) \cdot \Zp[[G_\infty]]^\sharp.$
\end{enumerate}
\end{theorem}
\begin{proof} Once again, we have $X_S^{(n)}=0$. Since $(\Bbb Z_p/d_S^{(n)}\Zp)^\sharp=0$ (as $\Delta$ acts trivially on
  $\Bbb Z_p/d_S^{(n)}\Zp$), when applying the exact functor $^\sharp$
to exact sequences \eqref{RW-ses-ell} and \eqref{tilde-x-ses-ell}, we obtain
\[
  \Nabla_S^{(n), \sharp} = (\Pic^0(L_n)\otimes\Bbb Z_p)^\sharp,
  \qquad \Nabla_S^{(\infty), \sharp}=\frak X_{p}^{(\infty),\sharp},
\]
and the result follows by projecting the equality in Theorem \ref{limit theorem 2} onto $\Bbb Z_p[[G_\infty]]^\sharp$.
\end{proof}

\begin{remark}\label{remark-Bandini-Coscelli}
     In this remark we use the assumptions of \cite{Bandini} and \cite{Coscelli}. More specifically, we assume
    $\frf = \fre$, $d_\infty = 1$ and     $p \nmid [H_\fre:k]$.
    Note that, for example, these assumptions hold in the case of the Carlitz module which is studied in \cite{Angles}.
    
    We let $G_\frp(L_0/k)$ and $I_\frp(L_0/k)$ denote the decomposition, respectively inertia group associated to $\frp$ in $L_0/k$.
    We observe that $\Delta = \Gal(L_0/k)$ has order prime to $p$ and let $\chi \in \hat\Delta$ be a character, such that
    $\chi|_{G_\frp(L_0/k)}$ is non-trivial. Since $I_\frp(L_0/k) = \Gal(L_0/H_\fre) \sseq G_\frp(L_0/k)$, the set of
    such characters $\chi$ includes the characters of type 2 as defined in \cite[Def.~3.1]{Bandini} or
    \cite[Def.~2.3.6]{Coscelli}. Since
    we only consider real ray--class fields, we do not see characters $\chi$ of type 1, as defined in loc.cit.
    Then, for all $n \in \mathbb{N} \cup \{\infty\}$ and $S = \{ \frp \}$, we have
    \begin{eqnarray*}
      \left( X_S^{(n)} \tensor_\Ze \Zp \right)^\chi &=& \left( D_S(L_n) \tensor_\Ze \Zp \right)^\chi \\
                                                       &\simeq& \Zp[\Gal(L_n/k) / G_\frp(L_n/k)]^\chi \\
                                                       &\simeq& \Zp[\Delta / G_\frp(L_0/k)]^\chi = 0.
    \end{eqnarray*}
    As a consequence, we have $\Nabla_S^{(\infty),\chi}=\frak X_{p}^{(\infty), \chi}$ and our Theorem \ref{limit theorem 2}(2) implies that 
    \[
      \Fitt_{\Zp(\chi)[[\Gamma_\infty]]}\left(\frak X_{p}^{(\infty), \chi}\right) =
      \Theta_{S, \Sigma}^{(\infty)}(1, \chi) \cdot \Zp(\chi)[[\Gamma_\infty]],
    \]
    for all characters $\chi$ as above. Thus we recover the central results of \cite[Thm.~1.1]{Angles},
      \cite{Bandini} and \cite[Thm.~2.4.8]{Coscelli} restricted to the real Iwasawa towers considered in loc.cit.
\end{remark}

\subsection{$\Theta_{S, \Sigma}^{(\infty)}(1)$ is a non-zero divisor}
The goal of this section is a proof of part (1) of Theorem \ref{limit theorem 2} above. We will first establish a structure theorem for
the Iwasawa algebra $\Bbb Z_p[[G_\infty]]$ whose  proof will be based  on the following result on pro-$p$ groups. 
\begin{theorem}[Theorem 3.1 of \cite{Kiehlmann}] \label{pro-p-theorem}
  Let $\mathcal G$ be a pro-$p$ group with countable (topological) basis, whose torsion subroup $t(\mathcal G)$ has bounded exponent.
  Then $t(\mathcal G)$ is a closed subgroup of $\mathcal G$ and we have an isomorphism of topological groups
$$\mathcal G\simeq t(\mathcal G)\times\Bbb Z_p^X,$$
where $X$ is a cardinal in the set $\Bbb N\cup\{\aleph_0\}$. 
\end{theorem}

Here is the promised structure theorem for the Iwasawa algebra $\Bbb Z_p[[G_\infty]]$.

\begin{prop}\label{G-infty-algebra-prop} The following hold.
\begin{enumerate}
\item There are closed subgroups $\widetilde{G_0}$ and $\widetilde{\Gamma_\infty}$ of $G_\infty$, such that 
$$G_\infty=\widetilde{G_0}\times\widetilde{\Gamma_\infty},$$
with $\widetilde{\Gamma_\infty}\simeq \Bbb Z_p^{\aleph_0}$ (topologically), $[\Gamma_\infty:\Gamma_\infty\cap\widetilde{\Gamma_\infty}]<\infty$,  $[\widetilde{\Gamma_\infty}:\Gamma_\infty\cap\widetilde{\Gamma_\infty}]<\infty$, and $\widetilde{G_0}$ is isomorphic  to a subgroup of $G_0$.\\

\item There is an injective morphism of topological $\Bbb Z_p$--algebras
$$\Bbb Z_p[[G_\infty]]\simeq \Bbb Z_p[\widetilde{G_0}][[\widetilde{\Gamma_\infty}]]\hookrightarrow \bigoplus_{\rho\in\widehat{\widetilde G_0}(\Bbb Q_p)}\Bbb Z_p(\rho)[[\widetilde{\Gamma_\infty}]]\simeq \bigoplus_{\rho\in\widehat{\widetilde G_0}(\Bbb Q_p)}\Bbb Z_p(\rho)[[X_1, X_2, \dots]],$$
where the injective map in the middle is given by the usual $\rho$--evaluation maps.
\end{enumerate}

\end{prop}
\begin{proof} Part (2) is a clear consequence of part (1) and Proposition \ref{big-Iwasawa-algebra-prop}. 
For the proof of (1), note that, with notations as in \S3.2 above, if we let $P_\infty:=\varprojlim_n P_n$, then we have 
\begin{equation}\label{iso-P-infty}G_\infty\simeq P_\infty\times \Delta, \qquad P_\infty/\Gamma_\infty\simeq P_0.\end{equation}
Recall that $\Delta$ is the complement of the $p$--Sylow subgroup $P_n$ of $G_n$, for all $n\geq 0$. Since $\Gamma_\infty$ is torsion free, the second isomorphism above implies that 
$t(P_\infty)$ is isomorphic to a subgroup of $P_0$, therefore it is finite and, obviously, of bounded exponent. Since $\Gamma_\infty$ has countable basis, the second isomorphism above implies that $P_\infty$ has countable basis as well. 
Consequently, Theorem \ref{pro-p-theorem} applied to $\mathcal G:=P_\infty$ gives a topological isomorphism 
\begin{equation}\label{P-infty-structure}P_\infty\simeq t(P_\infty)\times\widetilde{\Gamma_\infty},\end{equation}
where $\widetilde{\Gamma_\infty}\simeq \Bbb Z_p^{\aleph_0}$. Consequently, we have 
$$G_\infty=\widetilde{G_0}\times\widetilde{\Gamma_\infty}, \qquad \text{for  }\widetilde{G_0}:=t(G_\infty)\simeq t(P_\infty)\times\Delta.$$
Obviously, $\widetilde{G_0}$ is isomorphic to a subgroup of $G_0\simeq P_0\times\Delta$.  The fact that $[\Gamma_\infty:\Gamma_\infty\cap\widetilde{\Gamma_\infty}]<\infty$ and  $[\widetilde{\Gamma_\infty}:\Gamma_\infty\cap\widetilde{\Gamma_\infty}]<\infty$ follows immediately from \eqref{iso-P-infty} and \eqref{P-infty-structure}.
\end{proof}

\begin{remark}\label{phi-remark}  From now on, we let
  \[
    \varphi=(\varphi_\rho)_\rho:\Bbb Z_p[[G_\infty]]\simeq\Bbb Z_p[\widetilde{G_0}][[\widetilde{\Gamma_\infty}]]
    \hookrightarrow
    \bigoplus_{{\rho\in\widehat{\widetilde G_0}(\Qp)}}\Bbb Z_p(\rho)[[\widetilde{\Gamma_\infty}]]=:\overline{\Bbb Z_p[[G_\infty]]}
    \]
denote the character evaluation  map described above. Proposition \ref{big-Iwasawa-algebra-prop}(2)--(3) implies that the direct summands of  $\overline{\Bbb Z_p[[G_\infty]]}$ are integral, normal domains.  (Notice that $\overline{\Bbb Z_p[[G_\infty]]}$ is the integral closure of $\Bbb Z_p[[G_\infty]]$ in its total ring of fractions, which justifies the notation.)
\end{remark}

\begin{prop}\label{not a zero divisor prop}
   $\Theta_{S, \Sigma}^{(\infty)}(1)$ is a non-zero divisor in $\overline{\Zp[[G_\infty]]}$ and therefore in $\Bbb Z_p[[G_\infty]]$.
\end{prop}
\begin{proof} Proposition \ref{G-infty-algebra-prop}(2) implies that the statement to be proved is equivalent to 
 $$\varphi_\rho(\Theta_{S, \Sigma}^{(\infty)}(1))\ne 0 \text{ in }\Bbb Z_p(\rho)[[\widetilde{\Gamma_\infty}]],$$
 for all characters $\rho\in\widehat{\widetilde{G_0}}$. Of course, this is equivalent to proving that for every $\rho$ as above, there exists a character $\psi$ of $\widetilde{\Gamma_\infty}$, of open kernel (so, of finite order), such
 that 
 $$\psi(\varphi_\rho(\Theta_{S, \Sigma}^{(\infty)}(1))\ne 0 \text{ in }\Bbb Z_p(\rho\psi).$$
 However, from the definition of $\Theta_{S, \Sigma}^{(\infty)}$, for all $\rho$ and $\psi$ as above, we have an equality
 $$\psi(\varphi_\rho(\Theta_{S, \Sigma}^{(\infty)}(1))=L_{S,\Sigma}({(\rho\psi)^{-1}}, 0),$$
 where $\rho\psi$ is viewed as a complex--valued character of the finite quotient  
 $$\widetilde{G_0}\times(\widetilde{\Gamma_\infty}/\ker\psi)=G_\infty/\ker\psi$$
 of $G_\infty$
 (under a fixed field isomorphism $\Bbb C_p\simeq\Bbb C$) and  $L_{S,\Sigma}(\rho\psi, s)$ is the complex--valued Artin $L$--function ($S$--imprimitive and $\Sigma$--completed) associated to $\rho\psi$. (See equalities \eqref{poly versus L fct} above.)
 \smallskip
 
  Now, the following Lemma is a well--known description of the order of vanishing at $s=0$ of the Artin $L$--functions in question.
 For the number field case of this result, see \cite[Ch.~I, Prop.~3.4]{Ta84}
   and for the function field case, relevant in our
   context, see \cite[Sec.~2.2]{Po11}.
 \begin{lemma}\label{order of vanishing} If $\chi$ is a non--trivial character of $G_\infty$ with open kernel, then 
 $${\rm ord}_{s=0}\, L_{S, \Sigma}(\chi, s)={\rm card}\,\{v\in S\mid \chi(G_v(L_\infty/k))=1\},$$
 where $G_v(L_\infty/k)$ denotes the decomposition group of $v$ inside $G_\infty=G(L_\infty/k)$.
 \end{lemma} 
 
 Consequently, we claim that it suffices to find a finite subextension $L/K$ of $L_\infty/L_\infty^{\widetilde{\Gamma_\infty}}$ and a character $\chi$ of $G(L/K)$, such that the following conditions are simultaneously satisfied.
 \begin{enumerate}
 \item[(A1)] $\widetilde{L_0}:=L_\infty^{\widetilde{\Gamma\infty}}\subseteq K\subseteq L\subseteq L_\infty$ and $L/\widetilde{L_0}$ finite.
 \item[(A2)] $\chi(G_v(L/K))\ne \{1\}$, for all $v\in S$.
 \end{enumerate}
 \medskip
  Indeed, if we construct an $L/K$ and $\chi$ as above, then for any given character $\rho$ of $\widetilde G_0\simeq G(\widetilde{L_0}/k)$, we take any character $\psi$ of $G(L/\widetilde{L_0})$,
   such that $\psi\mid_{G(L/K)}=\chi$. Now, $\rho\psi$ is a character of $G(L/k)$ which satisfies the property
 $$\rho\psi (G_v(L/K))=\psi(G_v(L/K))=\chi(G_v(L/K))\ne \{1\}, \qquad \text{ for all } v\in S.$$ 
 Since $G_v(L/K)\subseteq G_v(L/k)$, for all $v\in S$, Lemma \ref{order of vanishing} gives us the desired nonvanishing
 $$\psi(\varphi_\rho(\Theta_{S, \Sigma}^{(\infty)}(1))=L_{S,\Sigma}({(\rho\psi)^{-1}}, 0)\ne 0.$$

 Now, since 	$[\widetilde{\Gamma_\infty}:\widetilde{\Gamma_\infty}\cap\Gamma_\infty]<\infty$, the existence of $L/K$ satisfying conditions (A1)-(A2) above is ensured if we can find two integers
 $m>n$ and a character $\chi$ of $G(L_m/L_n)$, such that
   \begin{enumerate}
 \item[(B1)] $n$ is large enough, so that $L_\infty^{\Gamma_\infty\cap\widetilde{\Gamma_\infty}}\subseteq L_n$. (Note that $\widetilde{L_0}\subseteq L_\infty^{\Gamma_\infty\cap\widetilde{\Gamma_\infty}}$.)
 \item[(B2)] $\chi(G_v(L_m/L_n))\ne \{1\}$, for all primes  $v\in S$.
 \end{enumerate}
 
 Now, we proceed to constructing $m$ and $n$ as above. First, we fix an $n\geq 0$, large enough so that (B1) is satisfied. Now, we apply Proposition \ref{decomposition groups prop} to get topological group isomorphisms 
 \begin{equation}\label{decomposition groups infty}G_{\frak f}(L_\infty/L_n)=\prod_{v\vert\frak f}G_v(L_\infty/L_n)\simeq\prod_{v\vert\frak f}\Bbb Z_p,\end{equation}
 where the notations are as in loc.cit. Since the $p$--adic and the profinite topologies on $G_{\frak f}(L_\infty/L_n)$ coincide, there exists an $m>n$, such that 
 \begin{equation}\label{topologies} G_{\frak f}(L_\infty/L_n)\cap G(L_\infty/L_m)\subseteq p\cdot G_{\frak f}(L_\infty/L_n).\end{equation}
 We let $G_{\frak f}(L_m/L_n)$ denote the subgroup of $G(L_m/L_n)$ generated by the decomposition groups $G_v(L_m/L_n)$, for all $v\vert\frak f$. From the definitions, we have a group morphism 
 $$G_{\frak f}(L_m/L_n)\simeq \frac{G_{\frak f}(L_\infty/L_n)}{G_{\frak f}(L_\infty/L_n)\cap G(L_\infty/L_m)}\twoheadrightarrow \frac{G_{\frak f}(L_\infty/L_n)}{p\cdot G_{\frak f}(L_\infty/L_n)}=\prod_{v\vert\frak f}\frac{G_v(L_\infty/L_n)}{p\cdot G_v(L_\infty/L_n)}, $$
 where the isomorphism to the left is induced by Galois restriction, the surjection is induced by the inclusion \eqref{topologies} and the equality is a consequence of \eqref{decomposition groups infty}.

 It is easily seen that for all $v\vert\frak f$ the above morphism maps 
 $G_v(L_m/L_n)$ onto the quotient $G_v(L_\infty/L_n)/pG_v(L_\infty/L_n)$ which by Proposition \ref{decomposition groups prop} is isomorphic to $\Bbb Z/p\Bbb Z$.
 Consequently, there is a character $\psi$ of $G_{\frak f}(L_m/L_n)$, such that 
 $\psi(G_v(L_m/L_n))=\{\zeta\in\Bbb C_p\mid \zeta^p=1\}$ for all  $v\vert\frak f$.
 Now, take any character $\chi$ of $G(L_m/L_n)$ which equals $\psi$ when restricted to $G_{\frak f}(L_m/L_n)$. This character obviously satisfies (B2) for all $v\vert\frak f$. Since it is non--trivial on $G(L_m/L_n)$ and $G_{\frak p}(L_m/L_n)=G(L_m/L_n)$ (recall that $L_m/L_n$ is totally ramified at the $\frak p$--adic primes), 
 the character $\chi$ also satisfies (B2) for $v=\frak p$. This concludes the proof of Proposition 3.22.
\end{proof}

We conclude this section with a corollary to Proposition \ref{not a zero divisor prop}.

\begin{prop}\label{nabla torsion prop}
  The $\Bbb Z_p[[G_\infty]]$--module $\Nabla_S^{(\infty)}$ is finitely generated,  torsion, and of projective dimension $1$.
\end{prop}

\begin{proof} We will use the notations in the proof of Theorem \ref{limit theorem 1}. In particular, note that 
$$\Lambda^\chi/(1-\gamma)\simeq\Bbb Z_p(\chi)[[P_\infty]], $$
for all $\overline{\Bbb Q_p}$--valued characters $\chi$ of $\Delta$. Consequently, the exact sequence \eqref{Lambda-ses} leads to the following commutative diagram of $\Lambda$--modules.
\begin{equation*}
\xymatrix{
 &  &  & U_S^{(\infty)} \ar@{>->}[d] &  \\
0 \ar[r] &\bigoplus_\chi \left(\Lambda^\chi\right)^{m_\chi} \ar[r]^{\Phi_\gamma^{(\infty)}} \ar[d]^{\gamma-1} &  \bigoplus_\chi\left(\Lambda^\chi\right)^{m_\chi} \ar[r]^{\pi_\infty} \ar[d]^{\gamma-1} &  T_p(M_S^{(\infty)}) \ar[r] \ar[d]^{\gamma-1} & 0 \\
0 \ar[r] & \bigoplus_\chi\left(\Lambda^\chi\right)^{m_\chi} \ar[r]^{\Phi_\gamma^{(\infty)}} \ar@{>>}[d] &  \bigoplus_\chi\left(\Lambda^\chi\right)^{m_\chi}
\ar[r]^{\pi_\infty} \ar@{>>}[d] &  T_p(M_S^{(\infty)}) \ar[r] \ar@{>>}[d] & 0 \\
 & \bigoplus_\chi\Zp(\chi)[[P_\infty]]^{m_\chi}  & \bigoplus_\chi \Zp(\chi)[[P_\infty]]^{m_\chi}  & \Nabla_S^{(\infty)}  &  \\
}
\end{equation*} 
where the right vertical exact sequence is given by Lemma \ref{infty Tate seq lemma}.
The snake lemma applied to the diagram above gives the exact sequence of $\Bbb Z_p[[G_\infty]]$--modules
\begin{equation*}
\xymatrix{
\bigoplus_\chi\Zp(\chi)[[P_\infty]]^{m_\chi} \ar[r]^{\overline{\Phi_\gamma^{(\infty)}}} & \bigoplus_\chi\Zp(\chi)[[P_\infty]]^{m_\chi}  \ar[r] & \Nabla_S^{(\infty)} \ar[r] & 0,  \\
}
\end{equation*} 
where $\overline{\Phi_\gamma^{(\infty)}}:= {\Phi_\gamma^{(\infty)}} \mod (\gamma-1)$. It follows that
\[
\Fitt_{\Zp[[G_\infty]]} \left( \Nabla_S^{(\infty)} \right) = \det\nolimits_{\Zp[[G_\infty]]} \left( {\overline{\Phi_\gamma^{(\infty)}}} \right) \cdot \Zp[[G_\infty]].
\]
Combined with Theorem \ref{limit theorem 2}(2), the equality above shows that the element
$\det\nolimits_{\Zp[[G_\infty]]} \left( {\overline{\Phi_\gamma^{(\infty)}}} \right)$  differs from $\Theta_{S, \Sigma}^{(\infty)}(1)$
  by a unit in $\Bbb Z_p[[G_\infty]]$. By Propositon \ref{not a zero divisor prop}, the element $\Theta_{S,\Sigma}^{(\infty)}(1)$ is a non-zero divisor in
$\Bbb Z_p[[G_\infty]]$. Therefore, $\det\nolimits_{\Zp[[G_\infty]]} \left( {\overline{\Phi_\gamma^{(\infty)}}} \right)$ is a non-zero divisor in $\Bbb Z_p[[G_\infty]]$ as well.
By a standard argument (using the adjoint matrix $\chi$--componentwise, see proof of Cor. \ref{non-zero divisors}(2)) we see that  ${\overline{\Phi_\gamma^{(\infty)}}}$ is injective, hence the sequence of $\Bbb Z_p[[G_\infty]]$--modules
\begin{equation*}
\xymatrix{
0 \ar[r] & \bigoplus_\chi\Zp(\chi)[[P_\infty]]^{m_\chi}  \ar[r]^{\overline{\Phi_\gamma^{(\infty)}}} & \bigoplus_\chi\Zp(\chi)[[P_\infty]]^{m_\chi} \ar[r] & \Nabla_S^{(\infty)} \ar[r] & 0  \\
}
\end{equation*} 
is exact. Consequently, 
the $\Bbb Z_p[[G_\infty]]$--module $ \Nabla_S^{(\infty)}$ is finitely generated, of projective dimension at most one. Further, the fact that $ \Nabla_S^{(\infty)}$  is torsion and of projective dimension exactly $1$ as a $\Bbb Z_p[[G_\infty]]$--module follows immediately from Lemma \ref{torsion lemma}. 
\end{proof}

\section{Appendix ($p$--adic Ritter--Weiss modules and Tate sequences for small $S$)} Let $L$ be a finite,
separable extension of $\Bbb F_q(t)$. Denote by $Z$ a smooth, projective curve defined over $\Bbb F_q$, whose field of rational functions
is isomorphic to $L$. We let $\overline Z:=Z\times_{\Bbb F_q}\overline{\Bbb F_q}$, $\Gamma:=G(\overline{\Bbb F_q}/\Bbb F_q)$
and let $\gamma$ be the $q$--power arithmetic Frobenius automorphism, viewed as a canonical topological generator of $\Gamma$.
Note that $\overline Z$ may not be connected. Consequently, $\overline L:=L\otimes_{\Bbb F_q}\overline{\Bbb F_q}$ (the $\overline{\Bbb F_q}$--algebra of rational functions on $\overline Z$) could be a finite direct sum of isomorphic fields (the fields of rational functions of the connected components of $\overline Z$.)\\

Next, we consider a finite, non--empty set $S$ of closed points on $Z$ and let  $\overline S$ be the set of closed points on $\overline Z$ sitting above points in $S$. We let ${\rm Div}^0_{\overline S}(\overline L)$ (respectively ${\rm Div}_{\overline S}(\overline L)$) and ${\rm Div}^0_S(L)$ (respectively ${\rm Div}_S(L)$) denote the divisors of degree $0$ (respectively arbitrary degree) on $\overline Z$ and  $Z$, supported at $\overline S$ and $S$, respectively. Note that the degree of a divisor on $Z$, denoted by ${\rm deg}$, is computed relative to the field of definition $\Bbb F_q$. Also, the degree of a divisor on $\overline Z$ is in fact a multidegree, computed on each connected component on $\overline Z$ separately. Further, $X_S(L)$ denotes the $S$--supported divisors on $Z$ of arbitrary formal degree, denoted below ${\rm fdeg}$.
Also, $U_S(L)$ denotes the group of $S$--units inside $L^\times$ and 
\[
  {\rm Pic}^0_S(L):=\frac{{\rm Pic}^0(L)}{\widehat{{\rm Div}^0_S(L)}} = \frac{{\rm Div^0}(L)}{{\rm Div^0_S}(L) + {\rm div}(L^\times)},
\]
is the $S$--Picard group associated to $L$, obtained by taking the quotient of the usual Picard group ${\rm Pic}^0(L)$
by the subgroup ${\widehat{{\rm Div}^0_S(L)}}$ of classes of all  $S$--supported divisors of degree $0$.\\

Finally, we let $M_S$ denote the Picard $1$--motive associated as in \cite{GP12} to the data $(\overline Z, \overline{\Bbb F_q}, \overline S, \emptyset)$. As usual, $T_p(M_S)^\Gamma$ and $T_p(M_S)_{\Gamma}$ denote the $\Gamma$--invariants, respectively $\Gamma$--coinvariants of the $p$--adic Tate module of $M_S$. In what follows, if $N$ is a $\Bbb Z$--module, we let $N_p:=N\otimes_{\Bbb Z}\Bbb Z_p$.

\begin{definition}\label{p-large-definition} The set $S$ is called $p$--large if the following are satisfied.
\begin{enumerate}
\item
${\rm Pic}^0_S(L)_p=0$.
\item 
$S$ contains at least one place of degree (relative to $\Bbb F_q$) coprime to $p$. 
\end{enumerate}
\end{definition}

\begin{remark} It is easily seen that $S$ is $p$--large if and only if ${\rm Pic}_S(L)_p=0$, where 
$${\rm Pic}_S(L):=\frac{{\rm Div}(L)}{{\rm Div}_S(L) + {\rm div}(L^\times)}$$
is the quotient of the full Picard group ${\rm Pic}(L)$ of $L$ by its subgroup of $S$--supported divisor classes. This is perhaps a more natural definition, but we prefer to use the definition above because ${\rm Pic}^0(L)$ (as opposed to ${\rm Pic}(L)$) is much more naturally related to the $1$--motive $M_S$.
\end{remark}

The following result was obtained in \cite{GPff}. (See Proposition 1.1 in loc.cit.)

\begin{prop}[Greither--Popescu, \cite{GPff}]\label{largeS-proposition} If $S$ is $p$--large, then the following hold.
\begin{enumerate}
\item There is a canonical isomorphism $T_p(M_S)^{\Gamma}\simeq U_S(L)_p$.
\item There is a canonical isomorphism $T_p(M_S)_{\Gamma}\simeq X_S(L)_p$.
\end{enumerate}
\end{prop}

\begin{remark} In fact, in \cite{GPff}, the authors describe the modules $T_p(M_{S, T})^\Gamma$ and $T_p(M_{S, T})_\Gamma$, where $M_{S, T}$ is the Picard $1$--motive associated to $(\overline Z, \overline{\Bbb F_q}, \overline S, \overline T)$, where $T$ is a finite, non--empty set of closed points on $Z$, disjoint from $S$. However, by \cite[Rem.~2.7]{GP12} and the proof of \cite[Lemma 3.2]{GPff}, we have for any such $T$ equalities
\[
    T_p(M_{S, T})=T_p(M_S), \qquad U_{S, T}(L)_p=U_S(L)_p,
\]
where $U_{S,T}(L)$ is the group of $S$--units in $L^\times$, congruent to $1$ modulo all primes in $T$.
\end{remark}

The goal of this Appendix is to remove the hypothesis ``$S$ is $p$--large'' in the Proposition above. More precisely,
we sketch the proof of the following.
\begin{prop}\label{smallS-proposition} With notations as above, the following hold for all sets $S$.
\begin{enumerate}
\item There is a canonical isomorphism $T_p(M_S)^{\Gamma}\simeq U_S(L)_p$.
\item There are canonical exact sequences of $\Bbb Z_p$--modules:
  \begin{eqnarray*}
    && 0\to {\rm Pic}^0_S(L)_p\to T_p(M_S)_{\Gamma}\to \widetilde{X_S}(L)\to 0, \\
    && 0\to \Bbb Z_p/d_S\Zp \to \widetilde{X_S}(L)\to X_S(L)_p\to 0,
  \end{eqnarray*}
where $d_S\Bbb Z={\rm deg}({\rm Div}_S(L))$ and $\widetilde X_S(L):= ({\rm Div}^0_{\overline S}(\overline L)_p)_\Gamma$.  In particular, if $S$ contains a prime of degree not divisible by $p$, then $\widetilde{X_S}(L)=X_S(L)_p$.
\end{enumerate}
\end{prop}
\begin{proof} (Sketch) We will give only a brief sketch of the proof, as the techniques and main ideas are borrowed from \cite{GPff}. First, we consider the exact sequence of $\Bbb Z_p[[\Gamma]]$--modules
$$0\to {\rm Div}^0_{\overline S}(\overline L)_p\to {\rm Div}_{\overline S}(\overline L)_p\overset{\rm deg}\longrightarrow \Bbb Z_p\to 0$$
and take $\Gamma$--invariants and $\Gamma$--coinvariants to obtain a long exact sequence
$$0\to {\rm Div}^0_{S}(L)_p\to {\rm Div}_{S}(L)_p\overset{\rm deg}\longrightarrow \Bbb Z_p\to ({\rm Div}^0_{\overline S}(\overline L)_p)_\Gamma\to  {\rm Div}_{S}(L)_p\overset{\rm fdeg}\longrightarrow \Bbb Z_p\to 0.$$
The fact that the $\Gamma$--invariant of the complex $[{\rm Div}_{\overline S}(\overline L)_p\overset{\rm deg}\longrightarrow \Bbb Z_p]$ is $[{\rm Div}_{S}(L)_p\overset{\rm deg}\longrightarrow \Bbb Z_p]$ and its $\Gamma$--coinvariant is 
$[{\rm Div}_{S}(L)_p\overset{\rm fdeg}\longrightarrow \Bbb Z_p]$ follows immediately from the definitions and is explained in \S2 of \cite{GPff}. Now, in the long exact sequence above, we have
$${\rm ker}({\rm fdeg})=X_S(L)_p, \qquad {\rm coker}({\rm deg})=\Bbb Z_p/d_S\Zp.$$
Therefore, if we let $\widetilde X_S(L):= ({\rm Div}^0_{\overline S}(\overline L)_p)_\Gamma$, we have a canonical exact sequence
\begin{equation}\label{tilde-x-ses}
0\to \Bbb Z_p/d_S\Zp \to \widetilde{X_S}(L)\to X_S(L)_p\to 0.
\end{equation}

If $J$ denotes the Jacobian of $\overline Z$, there is a canonical exact sequence of $\Bbb Z_p$--modules 
$$0\to T_p(J)\to T_p(M_S)\to {\rm Div}^0_{\overline S}(\overline L)_p\to 0.$$
(See \S2 of \cite{GP12} for the exact sequence above.) Since we have a canonical isomorphism 
$$T_p(J)_\Gamma\simeq {\rm Pic}^0(L)_p$$
 (see Corollary 5.7 in \cite{GP12}) and $T_p(J)$ is $\Bbb Z_p$--free of finite rank, we also have 
 $$T_p(J)^{\Gamma}=0.$$
 Therefore, when taking $\Gamma$--invariants and $\Gamma$--coinvariants in the above exact sequence, we obtain a canonical long exact sequence of $\Bbb Z_p$--modules
$$0\to T_p(M_S)^\Gamma\to {\rm Div}^0_{S}(L)_p\overset{\delta}\longrightarrow {\rm Pic}^0(L)_p\to T_p(M_S)_\Gamma\to \widetilde X_S(L)\to 0,$$
where the connecting morphism $\delta$ is the usual divisor--class map. (See \cite[\S1]{GPff} for this fact.) Since there is a canonical isomorphism $U_S(L)_p\simeq {\rm ker}(\delta)$, where $U_S(L)_p$ injects into ${\rm Div}^0_{S}(L)_p$ via the divisor map, we obtain a canonical isomorphism of $\Bbb Z_p$--modules
$$T_p(M_S)^\Gamma\simeq U_S(L)_p,$$
which concludes the proof of part (1) of the Proposition.

To conclude the proof of part (2), observe that, by definition, we have ${\rm coker}(\delta)={\rm Pic}^0_S(L)_p$. Therefore, the last four non--zero terms of the long exact sequence above lead to a canonical short exact sequence of $\Bbb Z_p$--modules
\begin{equation}\label{coinvariant-ses} 0\to {\rm Pic}^0_S(L)_p\to T_p(M_S)_{\Gamma}\to \widetilde{X_S}(L)\to 0.\end{equation}
In combination with \eqref{tilde-x-ses}, this concludes the proof of part (2). 
\end{proof}

\begin{remark}\label{RW-Tate-remark} (Ritter-Weiss modules and Tate sequences.) Assume that $L$ is the top field in a finite, Galois extension $L/K$, of Galois group $G$ and that $\Bbb F_q(t)\subseteq K\subseteq L$. Further, assume that the set $S$ is $G$--equivariant.
Then, all the $\Bbb Z_p$--modules  involved in the proof of the above Proposition carry natural $\Bbb Z_p[G]$--module structures. Most importantly, due to their canonical constructions, all the exact sequences above are exact in the category of $\Bbb Z_p[G]$--modules. 

Exact sequence \eqref{coinvariant-ses} is the $p$--adic, function field analogue of the Ritter--Weiss exact sequence (see \cite{Ritter-Weiss}), defining a certain extension class $\Nabla_S$ of a module of $S$--divisors by an $S$--ideal class group, in the number field setting. This is what prompts the notation $\Nabla_S(L)_p:=T_p(M_S)_{\Gamma}$. 

Further, since $T_p(M_S)$ is $\Bbb Z_p[G]$--projective, the exact sequence of $\Bbb Z_p[G]$--modules
\begin{equation}\label{Tate-seq}0\to U_S(L)_p\to T_p(M_S)\overset{1-\gamma}\longrightarrow T_p(M_S)\to \Nabla_S(L)_p\to 0,\end{equation}
is the $p$--adic, function field analogue of a Tate exact sequence (see \cite{GPff} and also \cite{Ritter-Weiss} for more details), in the case where $S$ is not necessarily $p$--large. 

Of course, in order to cement these analogies, one would have to compute the extension classes of \eqref{coinvariant-ses} an \eqref{Tate-seq} in ${\rm Ext}^1_{\Bbb Z_p[G]}(\widetilde X_S(L), {\rm Pic}^0_S(L)_p)$ and 
${\rm Ext}^2_{\Bbb Z_p[G]}(\Nabla_S(L)_p, U_S(L)_p)$  and show that they coincide with the class--field theoretically meaningful Ritter--Weiss and Tate classes, respectively. In \cite{GPff}, this was done $\ell$--adically, for $\ell\ne p$, for the exact sequence \eqref{Tate-seq}, in the case where $S$ is $\ell$--large. (See Theorem 2.2 in loc.cit.) A proof of the $p$--adic analogue of that theorem (even in the case where $S$ is $p$--large) is still missing in the literature, unless $|G|$ is not divisible by $p$, in which case this was proved in \cite{GPff}. (See Theorem 2.2. in loc.cit.)
\end{remark}

\bibliographystyle{abbrv}
\bibliography{ref_bp}

\footnotesize

\bigskip

  Werner Bley, \textsc{Department of Mathematics, Ludwig-Maximilians-Universit\"at M\"unchen}\par\nopagebreak
  \textit{E-mail address: } \texttt{bley@math.lmu.de}

\medskip

  Cristian D.  Popescu, \textsc{Department of Mathematics, University of California San Diego}\par\nopagebreak
  \textit{E-mail address: }\texttt{cpopescu@math.ucsd.edu}

\end{document}